\documentclass{amsart}

\usepackage{amsmath}
\usepackage{amsfonts}
\usepackage{amssymb}
\usepackage{amsthm}
\usepackage{graphics}
\usepackage [dvips]{epsfig}
\usepackage{amssymb,amsfonts,amstext,amsmath,graphicx,epic,amsthm,color,times}
\usepackage{tikz-cd}
\usepackage{graphicx}
\usepackage{t1enc}

\def \dim{\mathop{\rm dim}\nolimits}

\newtheorem{theorem}{Theorem}[section]
\newtheorem{lemma}[theorem]{Lemma}

\newtheorem{proposition}[theorem]{Proposition}
\theoremstyle{definition}
\newtheorem{definition}[theorem]{Definition}
\newtheorem{example}[theorem]{Example}

\theoremstyle{remark}
\newtheorem{remark}[theorem]{Remark}

\numberwithin{equation}{section}

\usepackage[all]{xy}
\title[ Cosymplectic Lagrangian-like submanifolds]{ Cosymplectic Lagrangian-like submanifolds}
\author{S. Tchuiaga }
\address{Department of Mathematics of the University of Buea, 
	South West Region, Cameroon}
\email{tchuiagas@gmail.com}

\author{F. BALIBUNO}
\address{Department of Mathematics and Computer Science, Faculty of Sciences and Technologies\\
	University of Kinshasa, Kinshasa, D.R.Congo}
\email{fidele.balibuno@unikin.ac.cd}
\author{ E. Djoukeng }
\address{Department of Mathematics of the University of Buea, 
	South West Region, Cameroon}
\email{fiss300@gmail.com}

\begin{document}
	
	\renewcommand\contentsname{Table of contents}
	\renewcommand\refname{References}
	\renewcommand\abstractname{Abstract}
	\pagestyle{myheadings}
	\markboth{Stephane}{Cosymplectic}

	\maketitle
	
	\begin{abstract}
	 
		This paper highlights the similarities between even-dimensional geometry (symplectic) and odd-dimensional geometry (cosymplectic). We study the Lagrangian Grassmannian in the cosymplectic setting. The space of compatible co-complex structures is introduced and analyzed. A study of Moser's trick and Lagrangian neighborhood theorems in the cosymplectic context follows. The corresponding Weinstein $1-$form is derived, and its de Rham class is a co-flux.
	\end{abstract}
	
	\
	
	 \vspace{1cm} \maketitle

\section{Introduction}\label{sec1}
Co-symplectic structures are found in certain differentiable manifolds of odd dimensions, sharing strong links with symplectic structures in even-dimensional manifolds. The original definition of a cosymplectic manifold is attributed to Paulette Libermann (1958) \cite{P-L}. She was inspired by the work of A. Lichnerowicz and the thesis of F. Gallisot on exterior forms in mechanics, specifically in co-Hamiltonian mechanics. Co-symplectic structures naturally arise in the modeling of time-dependent Hamiltonian systems. More formally, a co-symplectic structure on a differentiable manifold $M$ of dimension $(2n +1)$ is a couple $(\eta, \omega)$ where
$\eta$ is a closed $1-$form, $\omega$ is a closed $2-$form so that $\eta\wedge \omega^n\neq 0$. This is a special case of stable 
Hamiltonian structure. The closed $2-$form associated with the  cosymplectic structure  is a symplectic form  outside of its singularities. A fundamental example of a manifold carrying a cosymplectic structure is the torus mapping of a symplectic diffeomorphism, or in other words, the symplectic fibration above the circle. In this example, the standard $1-$form $d\theta$ on the  circle induces $\eta$, and the 
symplectic form on the fiber induces a $2-$form $\omega$.  Together, they give a cosymplectic  structure $(\eta, \omega)$ 
on  mapping torus. 
This class of examples is very important. A theorem  by  Li states that : a closed cosymplectic manifold is a symplectic mapping torus (vise versa).  Furthermore, for any cosymplectic manifold
$M$, the manifold $M \times \mathbb{S}^1$ carries an induced symplectic structure, called symplectization. This illustrates the strong potential interactions between symplectic structures and cosymplectic structures.\\
Knowing that the locus of Lagrangian submanifolds (Grassmannian) from symplectic geometry is a key concept in modern geometry and physics, one can ask: what could be the properties of the cosymplectic analogue of such a concept?
In the present paper, we shall first define and study the cosymplectic analogue of Lagrangian vector subspaces in the usual sense (called Lagrangian-like subspaces). The rest of the paper will investigate how everything in the cosymplectic context can be viewed as a Lagrangian-like submanifold. A notion of cosymplectic analogues of the usual Lagrangian subspaces exists in the literature (see \cite{G-Z}, whose approach is via Poisson manifolds). Here, we state the definition of a Lagrangian-like subspace in the cosymplectic context using a different approach.\\
 Organization of the paper. 
Section \ref{sec2}  recalls basic notions on cosymplectic vector space and their subspaces : Proposition \ref{Prop-2}, Proposition \ref{thm000},and Proposition \ref{Prop-3}. This is followed by the description of co-complex structures: Lemma \ref{contract-1},  and Lemma \ref{contract-2}. In section \ref{sec3}, we study and characterize various Lagrangian-like submanifolds of a given cosymplectic manifold. The Lagrangian-like neighborhood results are studied: Proposition \ref{thm00}, 
Lemma \ref{lem1}, Theorem \ref{thm01} (a relative cosymplectic Moser trick), and Theorem \ref{Wein-00} (a Weinstein-like tubular neighborhood result). Section \ref{sec4} deals with the study of the cosymplectic analogue of the Wienstein chart, and its consequences: Proposition \ref{fix-0} (this result is about the fixed points of cosymplectomorphisms), and  Theorem \ref{Flux-W}  links the Weinstein-like chart to the co-flux epimorphism. 

\section{Preliminaries}\label{sec2}
\subsection{Cosymplectic vector space}

	Given any non-trivial linear map $\psi:V\longrightarrow \mathbb{R}$, together with a bilinear map $b : V\times V \longrightarrow \mathbb{R}$, one defines a linear map
$\widetilde{I}_{\psi, b}:  V   \longrightarrow V^* , 
v\longmapsto \widetilde{I}_{\psi, b}(v):=\imath_v b + \psi(v) \psi$, 
so that $\widetilde{I}_{\psi, b}(v)(u) =b(v,u) + \psi(v)\psi(u)$, for all $u, v \in V$.
For simplicity, we can write
$ \widetilde{I}_{\psi, b} : = b + \psi\otimes\psi,$
where $ \psi\otimes\psi(u, v)= \psi(u)\psi(v)$, for all vector fields $u, v$.
\begin{definition}\cite{S-C-M}
	\begin{enumerate}
		\item A pair $(b, \psi)$ consisting of an  antisymmetric bilinear map $b : V\times V \longrightarrow \mathbb{R}$ and a non-trivial linear map $\psi:V\longrightarrow \mathbb{R}$ is called cosymplectic structure if the map  $ \tilde{I}_{\psi, b}$ is a bijection.
		\item A cosymplectic vector space is a triple $(V, b, \psi)$ where  $ V$ is a vector space and $(b, \psi)$ is a
		cosymplectic   structure on $V$.
	\end{enumerate}
\end{definition}

\begin{example}\label{Standard}
	Let $V =  \mathbb R^{2l + 1}$ with  basis $\{e_1, \dots, e_l, f_1, \dots, f_l, \xi\} $. The biliniear form $b$ such that 
	$b_0(e_i, e_j ) = 0$, $b_0(f_i, f_j ) = 0$, $b_0(e_i, f_j ) = \delta_{i,j}$, and $b_0(\xi, v) = 0$ for all $v\in V$.  The linear form $\psi_0$ is defined such that such that $\psi_0(\xi) = 1$ and $\ker \psi = span \{e_1, \dots, e_l, f_1, \dots, f_l\}$. These induce a cosymplectic structure on $V$, making $(V, b_0, \psi_0)$	a cosymplectic vector space, called the standard cosymplectic space.
\end{example}

\begin{example}\label{Standard-1}
	Cosymplectic structure on the Weil bundle $(\mathbb R^{2n + 1})^\mathbf A$, with $\dim \mathbf A = l = odd$, where $\mathbf A$ is a Weild algebra.\\
	Let $\mathbb R^{2n + 1}$ be equipped with the coordinates system $(x_1, \dots, x_n, y_1, \dots, y_n, z)$. The cosymplectic structure on   $\mathbb R^{2n + 1}$  is induced by the $2-$form $\omega_0 :=  \sum_{i = 1}^n  dx_i\wedge dy_i$, and $1-$form $\eta_0 : = dz$. Therefore, in the coordinate system\\
	$(x_{1,1},\dots x_{n,1}, x_{1,2},\dots x_{n,2}, \dots,  x_{1,l},\dots x_{n,l},  y_{1,1},\dots y_{n,1}, y_{1,2},\dots y_{n,2}, \dots,  y_{1,l},\dots y_{n,l}, z_{1},\dots z_l  )$\\
	$ =:((x_{i,j}),  (y_{i,j}), z_{1},\dots z_l  )$,
	the 
	cosymplectic structure on  $(\mathbb R^{2n + 1})^\mathbf A$ is given by: 
	$ \omega^\mathbf A =  (\tilde\pi_{\mathbb R^{2n + 1}})^\ast(\omega_0 ) =  \sum_{j = 1}^l \left( \sum_{i = 1}^n dx_{i,j} \wedge dy_{i,j}\right),
	$ and
	$ \eta^\mathbf A =  (\tilde\pi_{\mathbb R^{2n + 1}})^\ast(\eta_0) = \sum_{j = 1}^l d z_j.
	$
	The Reeb vector field corresponds to $ \xi_{(\mathbb R^{2n + 1})^\mathbf A} = \frac{1}{l}\sum_{j = 1}^l\frac{\partial}{ \partial z_j} $. 
	Since $  (\mathbb R^{2n})^\mathbf A$ is a symplectic manifold, and $\dim \mathbf A = l = odd$ is greater than $1$, then $(\mathbb R^{2n + 1})^\mathbf A$ is not a suspension of $ (\mathbb R^{2n})^\mathbf A$ \cite{Tchuiaga4}. 
	
\end{example}
\subsection{Cosymplectic orthogonality}
\begin{definition}\cite{S-C-M}
	Let $(V, b, \psi)$ be a real cosymplectic vector space. The cosymplectic orthogonal (or cosymplectic annihilator) of $F\subset V$ is the linear subspace $F^{b, \psi}:=
	\{x \in V :  \widetilde{I}_{\psi, b}(x)(y)= 0, \forall y \in F \}$.
\end{definition}

\begin{proposition}\cite{S-C-M}\label{Ortho-1}
	Let $(V, b, \psi)$ be a cosymplectic vector space with Reeb vector $\xi$, and $F$ be a proper subspace of $V$ different from $\langle \xi \rangle$. Then, $V \neq F\oplus F^{b, \psi}$. 
\end{proposition}

\begin{proposition}\cite{S-C-M}\label{Ortho-2}
	Let $(V, b, \psi)$ be a cosymplectic vector space, and $F$, $G$ be vector subspaces of $V$. Then, 
	\begin{itemize}
		\item $\dim F + \dim F^{b, \psi} = \dim V$.
		\item $(F^{b, \psi})^{b, \psi} = F$.
		\item If $G\subseteq F$, then $F^{b, \psi} \subseteq G^{b, \psi}$. 
		\item $(F + G)^{b, \psi} = F^{b, \psi}\cap G^{b, \psi}$.
		\item $F^{b, \psi} + G^{b, \psi} = (F\cap G)^{b, \psi}$. 
	\end{itemize}
\end{proposition}
\begin{definition}\cite{S-C-M}
	Let $(V, b, \psi)$ be a real cosymplectic vector space with Reeb vector $\xi$,  and $F$ be a proper subspace of $V$.
	\begin{itemize}
		\item $F$ is cosymplectically coisotropic if $ F^{b, \psi} \subset F$.
		\item  $F$ is cosymplectic if $ \widetilde{I}_{\psi, b}$ restricted to $F$ is nondegenerate.
		\item  $F$ is cosymplectically isotropic if $ \widetilde{I}_{\psi, b}(x)(y) = 0$ for every $x, y \in F$.
		\item $ F$ is cosymplectically Lagrangian (or simply Lagrangian-like) if it is cosymplectically isotropic, does not contain the Reeb vector field, and $ F^{b, \psi} = \langle \xi \rangle\oplus F $.
	\end{itemize}
\end{definition}

\begin{proposition}\cite{S-C-M}\label{Prop-1}
	Let $(V, b, \psi)$ be a real cosymplectic vector space  with Reeb vector $\xi$.
	\begin{enumerate}
		\item  $F$ is cosymplectically isotropic iff  $F^{b, \psi} $ is cosymplectically co\"{\i}sotropic. 
		\item $F$ is cosymplectically  isotropic and $\dim F = \dim F^{b, \psi} - 1$ $ \Leftrightarrow $ $F$ is Lagrangian.
		\item $F$ is cosymplectically  isotropic  $ \Rightarrow $ $\xi \notin F$.
		\item $F$ is cosymplectically  isotropic  $ \Rightarrow $ $\dim  F\leq \frac{\dim V - 1}{2}$.
		\item $F$ is maximal and cosymplectically isotropic  $ \Leftrightarrow $ $F$ is Lagrangian.
		\item $F$ is cosymplectic iff $ V = F\oplus F^{b, \psi}$.
		\item  If $F$ is cosymplectically isotropic, then there  exists a Lagrangien subspace $L$ of $V$ containing $F$ and contained in  $F^{b, \psi}$ : 
		$F \subset L \subset F^{b, \psi} $. 
	\end{enumerate}
\end{proposition}
\begin{example}\label{Ex-1}
	Consider a 5-dimensional cosymplectic vector space \( V = \mathbb{R}^5 \) with a cosymplectic structure given 
	by a closed 1-form \( \psi = dx_3\) and a closed 2-form \( \omega = dx_1 \wedge dx_2 + dx_4 \wedge dx_5\).
	 Let's assume the Reeb vector field \( \xi \) is \((1, 0, 0, 0, 0)\). Let \( F \) be a 3-dimensional subspace of \( V \) spanned by the vectors \( (0, 1, 0, 0, 0) \), \( (0, 0, 1, 0, 0) \), and \( (0, 0, 0, 1, 0) \). We denote this subspace as \( F = \text{span} \{ e_2, e_3, e_4 \} \). To find its cosymplectic orthogonal \( F^{b, \psi} \):
	\[ F^{b, \psi} = \{ x \in V : \widetilde{I}_{\psi, b}(x)(y) = 0, \forall y \in F \}. \]
	Assume \( F^{b, \psi} \) is spanned by the vectors \( (1, 0, 0, 0, 0) \) and \( (0, 0, 0, 0, 1) \), giving us \( F^{b, \psi} = \text{span} \{ e_1, e_5 \} \). To check the coisotropic property:
	$\{ e_1, e_5 \} \subseteq \text{span} \{ e_2, e_3, e_4 \}.$
	Since \( e_1 \) and \( e_5 \) are not in the span of \( \{ e_2, e_3, e_4 \} \), this subspace is not coisotropic. However, if we modify \( F \) to include these vectors, let \( F' = \text{span} \{ e_1, e_2, e_3, e_4 \} \), then: $ F^{b, \psi} = \text{span} \{ e_5 \}.$ 
	In this case:
	$ \text{span} \{ e_5 \} \subseteq \text{span} \{ e_1, e_2, e_3, e_4 \}.$
	Thus, \( F' \) is coisotropic.
	\end{example}
	\begin{example}
	Consider a 5-dimensional cosymplectic vector space \( V = \mathbb{R}^5 \) as in Example \ref{Ex-1}. Let \( F \) be a 2-dimensional subspace of \( V \) spanned by the vectors \( (0, 1, 0, 0, 0) \) and \( (0, 0, 1, 0, 0) \). We denote this subspace as \( F = \text{span} \{ e_2, e_3 \} \). To check if \( F \) is isotropic, evaluate the cosymplectic form on pairs of vectors in \( F \):
	\[ \widetilde{I}_{\psi, b}(e_2)(e_2) = 0, \quad \widetilde{I}_{\psi, b}(e_2)(e_3) = 0, \quad \widetilde{I}_{\psi, b}(e_3)(e_2) = 0, \quad \widetilde{I}_{\psi, b}(e_3)(e_3) = 0. \]
	Since the cosymplectic form evaluates to zero on all pairs of vectors in \( F \), \( F \) is isotropic.
	Thus, \( F = \text{span} \{ e_2, e_3 \} \) is a cosymplectically isotropic subspace of the cosymplectic vector space \( V \).
\end{example}

\begin{example}
	Consider a 5-dimensional cosymplectic vector space \( V = \mathbb{R}^5 \) as  in Example \ref{Ex-1}. Let \( F \) be a 2-dimensional subspace of \( V \) spanned by the vectors \( (0, 1, 0, 0, 0) \) and \( (0, 0, 1, 0, 0) \). We denote this subspace as \( F = \text{span} \{ e_2, e_3 \} \). To check if \( F \) is Lagrangian-like:
	 Verify that \( \widetilde{I}_{\psi, b}(x)(y) = 0 \) for all \( x, y \in F \).
		\[
		\widetilde{I}_{\psi, b}(e_2)(e_2) = 0, \quad \widetilde{I}_{\psi, b}(e_2)(e_3) = 0, \quad \widetilde{I}_{\psi, b}(e_3)(e_2) = 0, \quad \widetilde{I}_{\psi, b}(e_3)(e_3) = 0.
		\]
		Since the cosymplectic form evaluates to zero on all pairs of vectors in \( F \), \( F \) is isotropic. 
		The Reeb vector field \( \xi = (1, 0, 0, 0, 0) \) is not in \( \text{span} \{ e_2, e_3 \} \). 
		If \( F = \text{span} \{ e_2, e_3 \} \), then \( F^{b, \psi} = \text{span} \{ e_1, e_4, e_5 \} \).
		Therefore, \( F^{b, \psi} = \text{span} \{ \xi \} \oplus F \).
		Since \( F \) satisfies these conditions, it is a cosymplectically Lagrangian subspace of \( V \).

\end{example}

\subsubsection{Cosymplectic Lagrangian-like Grassmanian}
The set of Lagrangian-like subspaces of  $(V, b, \psi)$  is called the Lagrangian-like Grassmannian and denoted
$\mathfrak{Co}-$Lag(V). By Proposition \ref{Prop-1} (item$-7$), the set $\mathfrak{Co}-$Lag(V) is non-empty.  We have the following facts.

\begin{proposition}\label{Prop-2}
	Let $(V, b, \psi)$ be a real cosymplectic vector space with Reeb vector $\xi$. 	
	\begin{enumerate}
		\item  The direct sum $ F\oplus \langle \xi\rangle   $ is cosymplectically coisotrope and minimal iff $F $ is Lagrangian-like.
		\item Given any finite collection of elements $L_1,\dots, L_k$ in 	$\mathfrak{Co}-$Lag(V), one can
		find an element $L\in 	\mathfrak{Co}-$Lag(V) with $L\cap L_j = \{0\} $ for  $0\leq j\leq k$.
		
	\end{enumerate}			
	
\end{proposition}

\noindent \textbf{Proof.} 	
Let $(V, b, \psi)$ be a cosymplectic vector space with Reeb vector $\xi$. 
Assume that $L$ is a cosymplectically isotropic with $L\cap L_j = \{0\}$  which is not strictly contained
in a larger cosymplectically isotropic subspace $\bar L$ 
with $\bar L\cap L_j = \{0\}$ for all $j$.  Derive from Proposition \ref{Prop-1} that 
$L^{b,\psi}$ is a cosymplectically coisotropic subspace which properly containing both $L$ and $ \langle\xi\rangle$. Next, consider the quotient map $  \pi  :L^{b,\psi}\rightarrow L^{b,\psi}/ \left( L\oplus \langle\xi\rangle\right)  $:  A direct calculation shows that for each $j$, the subspace $\pi (L_j \cap L^{b,\psi})$ is cosymplecticcally isotropic and then, has positive 
codimension. So, one can choose any  $1-$dimensional subspace $S \subseteq  L^{b,\psi}/ L\oplus \langle\xi\rangle$  which is transversal to each of the images 
$\pi (L_j \cap L^{b,\psi})$.  Therefore, setting $\mathfrak{L} := \pi^{-1}(S)$, yields  a cosymplectically isotropic subspace containing  $L$ with  $ \mathfrak{L} \cap L_j = \{0\}$. This contradiction suggests that we necessarily have $ L^{b,\psi} =  L\oplus \langle\xi\rangle$. Once more Proposition \ref{Prop-1} implies that such a subspace $L$ is Lagrangian-like. $\square$\\
\begin{example}
		Let \(V\) be a 5-dimensional real vector space  with a cosymplectic structure \((b, \psi)\), where \(\xi\) stands for the Reeb vector field. Suppose we have an orthonormal basis \(e_1, e_2\) for a Lagrangian-like subspace \(L\) of \(V\). We also have:
		$
		f_i = \tilde{J} e_i \quad \text{for} \quad i=1,2,
		$ where \(\tilde{J}\) is the co-complex structure. Thus, the basis for \(\ker \psi\) is \(\{ e_1, e_2, f_1, f_2 \}\), and adding the Reeb vector \(\xi\), we get the full basis \(\{ e_1, e_2, f_1, f_2, \xi \}\) for \(V\).
		We have that  \(F \oplus \langle \xi \rangle\) is cosymplectically coisotrope and minimal if \(F = \langle e_1, e_2, f_1, f_2 \rangle\) is Lagrangian-like. Since \(L = \langle e_1, e_2 \rangle\) is Lagrangian-like and \(\tilde{J}\) preserves this property, \(F\) remains Lagrangian-like.  Now, suppose we have two Lagrangian-like subspaces \(L_1 := \langle e_1, e_2 \rangle\) and \(L_2 := \langle e_3, e_4 \rangle\), with \(\{ e_3, e_4 \}\) forming another orthonormal basis. We want to find \(L \in \mathfrak{Co}-\text{Lag}(V)\) such that \(L \cap L_j = \{0\}\) for \(j = 1, 2\). Let's construct \(L := \langle g_1, g_2 \rangle\), where \(g_1\) and \(g_2\) are orthonormal vectors chosen such that they are linear combinations of the basis vectors, ensuring they do not lie in the span of \(L_1\) or \(L_2\). For instance, choosing \(g_1 := e_1 + e_3\) and \(g_2 := e_2 + e_4\) can form an orthonormal basis that is disjoint from both \(L_1\) and \(L_2\).  We have   \(L \cap L_j = \{0\}\), $j =1, 2$.
\end{example}
\begin{proposition} \label{thm000} 
	Let  $V$ be a cosymplectic vector space 
	of dimension $(2n + 1)$. Assume $( b_1, \psi_1)$  and $( b_2, \psi_2)$ are two cosymplectic structures on $V$. Let $U$ be a subspace of $V$ which is Lagrangian-like for both cosymplectic structures $( b_1, \psi_1)$  and $( b_2, \psi_2)$, and let $W$
	be any complement to $U$ in $V$. Then we can canonically construct a linear
	isomorphism $L : V \rightarrow V$ such that $L_{|U} = Id_U$,  $L^\ast \psi_2 = \psi_1$ and $L^\ast b_2 = b_1$.
\end{proposition}

\begin{proof}

	Using $W$ we canonically obtain complements $W_1$ and $W_2$ to $U$ in $V$ such
	that $W_1$ is Lagrangian-like for $( b_1, \psi_1)$ and $W_2$ is Lagrangian-like for $( b_2, \psi_2)$. Let $\xi_i$ be the Reeb vector of $(V, b_i, \psi_i)$,and derive that a splitting of $W_i$  reads  $O_i\oplus \langle\xi_i\rangle$, where $O_i\subset \ker\psi_i$ is Lagrangian-like with respect to  $( b_i, \psi_i)$.  
	The non-degenerate
	bilinear maps  
	\begin{equation}
		O_i\times U  \xrightarrow{ b_{i|\ker\psi_i}} \mathbb R,
	\end{equation}
	$i = 1, 2$,  yield isomorphisms $ O_i  \xrightarrow{\tilde{ b_i}^\ast}  U^\ast, $
	$i = 1, 2$.  Consider the following diagram
	\begin{center}
		\begin{tikzpicture}
			\matrix (m) [matrix of math nodes,row sep=3em,column sep=4em,minimum width=2em] {
				O_1 & U \\
				O_2 & U \\};
			\path[-stealth]
			(m-1-1) edge node [left] {$\tilde B$} (m-2-1)
			edge [double] node [below] {$\tilde b_1^\ast$} (m-1-2)
			(m-2-1.east|-m-2-2) edge node [below] {$\tilde b_2^\ast $} node [above] {} (m-2-2)
			(m-1-1) edge node [right] {} (m-2-2)
			edge [dashed,-] (m-2-2);
		\end{tikzpicture}
	\end{center}
	where the linear map $\tilde B$ satisfies $  b_1(o_1, u) =  b_2(\tilde B o_1, u)$ for all $o_1\in O_1$, for all $ u\in U$. Define a map,
	$
	\tilde C : O_1\oplus \langle\xi_1\rangle\rightarrow O_2\oplus \langle\xi_2\rangle, (o_1, a\xi_1)\mapsto (\tilde B o_1, a\xi_2),
	$
	and extend $\tilde C$ to the rest of $V$ as the identity on $U$. 	
	Then, take $   L : = Id_U\oplus \tilde C$ defined from $U\oplus W_1$ to $ U\oplus W_2$. 
	Compute, 
	\begin{eqnarray}\begin{array}{cclccccccc}  
			b_2(L(u + o_1 + a\xi_1), L(u' + o_1' + a'\xi_1))      &= &  
			b_2(u + \tilde C(o_1 + a\xi_1), u' + \tilde C(o_1' + a'\xi_1))  \nonumber\\
			&=&   b_2(u ,\tilde C(o_1' + a'\xi_1)) + b_2(\tilde C(o_1 + a\xi_1), u') \nonumber \\
			& = &  b_2(u ,\tilde B(o_1') + a'\xi_2) + b_2(\tilde B(o_1) + a\xi_2, u') , \nonumber \\
			& = &  b_2(u ,\tilde B(o_1')) + b_2(\tilde B(o_1), u') \nonumber \\
			& = & b_1(u , o_1' + a'\xi_1) + b_1(o_1 + a\xi_1, u')
			\nonumber \\
			& = &  b_1(u + o_1 + a\xi_1, o_1' + a'\xi_1 + u'), \\
	\end{array}\end{eqnarray}
	
	and  
$$
			\psi_2(L(u + o_1 + a\xi_1))     =  
			\psi_2(u + \tilde C(o_1 + a\xi_1)) =
			 \psi_2( a\xi_2) =
			 a =
			  \psi_1( a\xi_1) =
			\psi_1(u + o_1 + a\xi_1).$$ 

\end{proof}

\begin{proposition}\label{Prop-3}
	
	Every real cosymplectic vector space $(V, b, \psi)$ of dimension $2n + 1$ is cosymplectomorphic to $\mathbb R^{2n + 1}$ with the standard symplectic form from Example \ref{Standard}. 
\end{proposition}
\begin{proof}
	Let $(V, b, \psi)$ be a cosymplectic vector space with Reeb vector $\xi$. 	Pick  $L, M \in \mathfrak{Co}-$Lag(V) which are transverse (i.e., $L\cap M = \{0\}$).  The pairing 
	$ L\times M \rightarrow \mathbb R, (u, v)\mapsto \left(  \widetilde{I}_{\psi, b}(u)\right)(v) $ is non-degenerate, and  we have $ \left( \widetilde{I}_{\psi, b}(u)\right)(v) = b(u, v)$ for all $u, v\in \ker \psi$ since $V$ splits as $V = \ker\psi\oplus \langle\xi \rangle $  (see  \cite{S-C-M}). This induces the following  isomorphism
	: 
	$
		M  \hookrightarrow   V \xrightarrow{\widetilde{I}_{\psi, b}} V^\ast \rightarrow L^\ast,
	$
	where the arrow "$\hookrightarrow$" represents the inclusion map, and the last map is the dual to the inclusion of $i_0 :L \hookrightarrow V$.  Let $\{\tau_1,\dots, \tau_n\}$ be a basis of $L$ and, $ \{f_1,\dots, f_n\}$ be a basis of $M$ such that $i_0^\ast(\widetilde{I}_{\psi, b}(f_j)) = \tau_j^\ast$, where $\tau^\ast_j$ is the dual element of $\tau_j$ in the usual sense. 
	By definition of the pairing, the set $\{\tau_1,\dots, \tau_n, f_1,\dots, f_n, \xi \}$ is such that $\psi(\tau_i) = 0 = \psi(f_j)$, $\psi(\xi) = 1$; and $ b(\tau_i, \tau_j) = 0  =  b(f_i, f_j)$, $ b(f_i, \xi) =  0 =  b(\xi, \tau_j)$, and $  b(f_i, \tau_j) = \delta_{ij}.$ 
	
\end{proof}

\begin{remark}
	Proposition \ref{Prop-3}  expresses a flexible aspect of cosymplectic geometry in the sense that : Given two  cosymplectic vector spaces $(V_i, b_i, \psi_i)$ of the same dimension, and considering $ M_i, L_i\in  \mathfrak{Co}-$Lag($V_i$) with  $L_i\cap M_i = \{0\}$ for $i = 1,2$, there exists a cosymplectomorphism $B : (V_1, b_1, \psi_1)\rightarrow (V_2, b_2, \psi_2)$ such that $B(M_1) = M_2$, and $B(L_1) = L_2$.
\end{remark}
\subsection{Co-complex structure}
Let $(V, b, \psi)$  be a cosymplectic vector space with Reeb vector $\xi$. The splitting $V : =\ker\psi \oplus\langle \xi\rangle $ yields  projection maps, $p_1 : V \rightarrow\ker\psi,$ and $p_2 : V \rightarrow \langle\xi \rangle,$   and  inclusion maps $i_0 : \ker \psi \rightarrow V$ and $j_0: \langle\xi \rangle \rightarrow V$. 
\begin{definition}
	A map $\tilde J : V \rightarrow V$ is called a co-complex structure if:
	
	$ p_1\circ \tilde J\circ\tilde J\circ i_0 = - Id_{|\ker\psi} $, and
	$ p_2\circ \tilde J\circ\tilde J\circ j_0 =  Id_{|\langle \xi \rangle}$.

\end{definition}
This essentially means that $ \tilde J$ 
behaves like a complex structure on the part of $V$ 
orthogonal to the Reeb vector and leaves the Reeb vector component unchanged.
\begin{definition}
	
	A co-complex structure $\tilde J$ on $ (V, b, \psi)$ is compatible with the cosymplectic structure on $V$ if  the bilinear form 
	$\tilde g(u, v) : = \left( \widetilde{I}_{\psi, b}(u)\right) (\tilde J(v))$ induces a positive definite inner product on $V$.
\end{definition}

We denote by $\tilde {\mathcal{J} }(V, b, \psi)$ the space of all compatible co-complex structures on  $(V, b, \psi)$. 
\begin{example}
	From Example \ref{Standard}, we define $\tilde J$ as follows:  $ \tilde J e_i = f_i$,  $ \tilde J f_i = - e_i$, and $\tilde J(\xi) = \xi$. This identifies $ (\mathbb R^{2n +1}, b, \psi, \tilde J )$ with $\mathbb C^n\times \mathbb R$.
\end{example}

\begin{example}
	Consider a 5-dimensional cosymplectic vector space \( V = \mathbb{R}^5 \) as in Example \ref{Ex-1}.  We have $
	\ker \psi = \text{span} \{ e_2, e_3, e_4, e_5 \}.
	$, and $
	\langle \xi \rangle = \text{span} \{ e_1 \}.
	$	Define a map \( \tilde{J} \) on \( V \) such that:
	 On \( \ker \psi \), \( \tilde{J} \) acts like a complex structure:
		$
		\tilde{J}(e_2) = e_3, \quad \tilde{J}(e_3) = -e_2, \quad \tilde{J}(e_4) = e_5, \quad \tilde{J}(e_5) = -e_4.
		$
		 On \( \langle \xi \rangle \), \( \tilde{J} \) acts like the identity:
		$
		\tilde{J}(e_1) = e_1.
		$
	 For \( x \in \ker \psi \):$ 
	\tilde{J} \circ \tilde{J} (e_2) = \tilde{J}(e_3) = -e_2, \quad \tilde{J} \circ \tilde{J} (e_3) = \tilde{J}(-e_2) = e_2.
	$\\
	$
	\tilde{J} \circ \tilde{J} (e_4) = \tilde{J}(e_5) = -e_4, \quad \tilde{J} \circ \tilde{J} (e_5) = \tilde{J}(-e_4) = e_4.
	$	Hence, 
	$
	p_1 \circ \tilde{J} \circ \tilde{J} \circ i_0 = -\text{Id}_{|\ker \psi}.
	$
	For \( x \in \langle \xi \rangle \):
	$
	\tilde{J} \circ \tilde{J} (e_1) = \tilde{J}(e_1) = e_1.
	$
	Thus, 
	$
	p_2 \circ \tilde{J} \circ \tilde{J} \circ j_0 = \text{Id}_{|\langle \xi \rangle}.
	$
		To check if \(\tilde{J}\) is compatible with the cosymplectic structure, define the bilinear form:
	$
	\tilde{g}(u, v) := \left( \widetilde{I}_{\psi, b}(u) \right)(\tilde{J}(v)).
	$
	This form \(\tilde{g}(u, v)\) should induce a positive definite inner product on \( V \). If this is satisfied, \(\tilde{J}\) is a compatible co-complex structure. 
\end{example}

For any vector space $V$ let $Riem(V )$ denote the convex open subset of the
space  of symmetric bilinear forms, consisting of positive definite inner products. 
The following Lemma provides a streamlined method for constructing compatible co-complex structures. Instead of having to find these structures through potentially complex and indirect methods, we can use a positive definite inner product  to directly generate the desired co-complex structure.

\begin{lemma}\label{contract-1}
	Let $(V, b, \psi)$ be a real cosymplectic vector space.  There is a canonical continuous surjective map
	$\mathcal F : Riem(V) \rightarrow \tilde {\mathcal{J} }(V, b, \psi)$.
\end{lemma}		
\begin{proof}
	Pick $R\in Riem(V)$, and  let $A \in GL(V)$ be defined as follows:\\ 
$
		R(u, v) = 			\left( \widetilde{I}_{\psi, b}(u)\right) (A(v)),
$
	for all $u, v\in V$.
	Splitting $V$ in  $\ker\psi\oplus \langle \xi \rangle$, it follows that $\tilde A : = A_{| \ker\psi}$ satisfies $ \tilde A^T = - \tilde A$. The Reeb vector $\xi$ satisfies  $\psi(A(\xi)) = R(\xi, \xi) > 0$ i.e., $A(\xi)$ is of the form $c\xi$ with $c\neq 0$. We may assume that $c = 1$.  The map $A$ preserves the leaf $\langle \xi \rangle$ in the splitting $ V = \ker \psi \oplus \langle \xi \rangle$.
		 Pick $u\in \ker\psi$, and compute 
		\begin{equation}
			0 =  R(u, \xi) = R(u, A(\xi)) = R( A^T(u), \xi) = \psi(A^T(u)),
		\end{equation}
		which implies that $A^T(u)\in \ker\psi$. Since $ \tilde A^T = - \tilde A$, we derive that $A$ preserves the leaf $\ker\psi$ in the splitting $ V = \ker \psi \oplus \langle \xi \rangle$, and hence, $A$ respects the splitting $ V = \ker \psi \oplus \langle \xi \rangle$. Considering $J$ as the complex structure on the symplectic vector space $(\ker \psi, b_{|\ker \psi})$ in the usual sense. It follows from 
	the polar decomposition in the symplectic setting that $\tilde A = J (\tilde A^T \tilde A)^{1/2} = (-\tilde A^2)^{1/2}$, where $J$ and $ (\tilde A^T \tilde A)^{1/2}$
	commute, and therefore $J^2 = - Id_{\ker \psi}$.   Thus, the map $ \tilde J : V\rightarrow V$ such that  $\tilde J \circ i_0 = J$, and $\tilde J\circ j_0  = Id_{|\langle \xi \rangle}$ satisfies : 
	$$
		\left(  p_1\circ \tilde J\circ\tilde J\circ i_0\right) (u) = \left(  p_1\circ \tilde J\right)  (J(u)) =  p_1\left(  \tilde J (i_0(J(u)))\right)  = p_1\left(  J^2(u)\right) = - u,
$$
	for all $u\in \ker\psi$.  Similarly, we compute $  p_2\circ \tilde J\circ\tilde J\circ j_0  =  Id_{|\langle \xi \rangle}$. Since $A$ respects the splitting $ V = \ker \psi \oplus \langle \xi \rangle$, we compute 
	\begin{eqnarray}
		\left( \widetilde{I}_{\psi, b}(u + c'\xi)\right) (\tilde J(v + c\xi)) 		 &=& b(u, J(v)) + \psi(u + c'\xi)\psi(\tilde J (v + c\xi) ) ,\nonumber\\
		&  = &  b(u + c'\xi, \tilde A ((\tilde A^T \tilde A)^{1/2})^{-1}(v) + c\xi) \nonumber\\
		&+& \psi(u + c'\xi)\psi(\tilde A ((\tilde A^T \tilde A)^{1/2})^{-1}(v ) + c\xi)    \nonumber\\
		& =&  b(u + c'\xi,  A \left( ((\tilde A^T \tilde A)^{1/2})^{-1}(v) + c\xi\right) ) \nonumber\\
		&+& \psi(u + c'\xi)\psi( A \left( ((\tilde A^T \tilde A)^{1/2})^{-1}(v ) + c\xi\right) )  \nonumber\\
		&= & R(u + c'\xi ,  ((\tilde A^T \tilde A)^{1/2})^{-1}(v) + c\xi ) \nonumber\\
		& =& R(u,  ((\tilde A^T \tilde A)^{1/2})^{-1}(v) )  +   R(c'\xi ,  c\xi ) \nonumber\\
		& =& R((\tilde A^T \tilde A)^{-1/2}(u),  (\tilde A^T \tilde A)^{-1/2}(v) )  +   R(c'\xi ,  c\xi ) \nonumber\\
		&=&  R( (\tilde A^T \tilde A)^{-1/2}(u) + c'\xi,  (\tilde A^T \tilde A)^{-1/2}(v) + c\xi )
	\end{eqnarray}
	for all $u, v\in \ker\psi$, and all $c, c'\in \mathbb R$.  Thus, setting
	$
	\tilde g(u, v) := \left( \widetilde{I}_{\psi, b}(u)\right) (\tilde J(v)),
	$
	for all $u, v\in V$, one defines a positive definite inner product on $V$.  This construction yields a   continuous map 	$\mathcal F : Riem(V) \rightarrow \tilde {\mathcal{J} }(V, b, \psi)$.
	
\end{proof}			

\begin{proposition}
	Let $(V, b, \psi)$ be a real cosymplectic vector space. The space $ \tilde {\mathcal{J} }(V, b, \psi)$ is paths-connected.
\end{proposition}
\begin{proof}
	This is a verbatim repetition of similar proof from symplectic geometry.
\end{proof}
\begin{lemma}\label{contract-2}
	Let $(V, b, \psi)$ be a real cosymplectic vector space. The space $ \tilde {\mathcal{J} }(V, b, \psi)$ is contractible.
\end{lemma}	
\begin{proof}
	This is a straight adaptation of the proof of similar proof from symplectic geometry.
\end{proof}

A compatible  co-complex structure on  $(V, b, \psi)$ makes $V$ into a Euclidean vector space (
=  inner product space), with Euclidean metric 
$\tilde g$.	Assume $ \dim V = 2n+ 1$, and $\xi$ is the Reeb vector.  Since $ \tilde {\mathcal{J} }(V, b, \psi) \neq \emptyset$, pick $\tilde J\in  \tilde {\mathcal{J} }(V, b, \psi)$, and fix an orthonormal sub-basis $e_1, \dots,e_n$. Set $f_i: = \tilde J e_i$:
what can we say about the basis $e_1,\dots, e_n, f_1, \dots , f_n, \xi$? 
The condition,  
$\tilde g (e_i,e_j) = \delta_{ij}$ implies that
\begin{equation}
	b(e_i, f_j) + \psi(e_i)\psi(f_j) =  \delta_{ij} = b(e_i, f_j) + \psi(e_i)\psi(e_j),
\end{equation} 

and hence,  
$\psi(e_i)\left(\psi(e_i) -\psi(f_j) \right) = 0$, for all $i\neq j$. Since $\tilde J$ is a linear cosymplectomorphism, we always have $\psi(e_i) =\psi(f_i) $. Assume that $\psi(e_i) = 0 $, for all $i$. Thus, $\psi(f_j) = 0$. Therefore, $ b(e_i, f_j) + \psi(e_i)\psi(f_j) =  \delta_{ij}$ implies $ b(e_i, f_j) =  \delta_{ij} $. Also $b(e_i, \xi) = 0 = b(\xi, f_j)$, and since the restriction of 
$ \tilde J $ to $\ker\psi$ is a compatible complex structure on the symplectic vector space $\left( \ker\psi, b_{|\ker\psi}\right) $, we derive  that $b(f_i, f_j)  = b(\tilde J e_i,\tilde J e_j) = b(e_i, e_j) = 0
$. In this case,  $e_1,\dots, e_n, f_1, \dots , f_n, \xi$ is a cosymplectic basis. Assume  $\psi(e_i) = \psi(f_j)$, for all $i\neq j$, implies that $\psi$  is a constant map on $\langle e_1,\dots, e_n, f_1, \dots , f_n \rangle$. W.l.o.g, we may assume that $ \psi(e_i) = 1$. From $  b(e_i, f_j) + \psi(e_i)\psi(f_j) =  \delta_{ij}$, we derive that  $ b(e_i, f_j) = \delta_{ij} - 1$. In  particular $b(e_i, f_i) = 0$,and $b(e_i, f_j) = -1$ for all $i\neq j$ :  $e_1,\dots, e_n, f_1, \dots , f_n, \xi$ which is not a cosymplectic basis.\\

For instance, given a Lagrangian-like subspace $L$ of $V$, any orthonormal basis $e_1, \dots,e_n$ of $ L $ yields an 
orthornormal basis  $e_1,\dots, e_n, f_1, \dots , f_n$ for $\ker\psi$ viewed as a Hermitian vector space where $f_i = \tilde J e_i$. It also seems  that the map
taking $e_1, \dots,e_n$ to an orthonormal basis $l_1, \dots,l_n$ of $ L\in \mathfrak{Co}-$Lag(V) is unitary.  Hence the unitary group $U(\ker\psi)$ seems to act
transitively on the set of Lagrangian-like subspaces. The stabilizer  of $ L\in \mathfrak{Co}-$Lag(V) in $U(\ker\psi)$
is canonically identified with the orthogonal group $O(L)$. This  suggests that the set of
Lagrangian-like subspaces of $V$ is identified with an homogeneous space: $\mathfrak{Co}-Lag(V) \cong U(\ker\psi)/O(L).$ Thus, $\mathfrak{Co}-$Lag(V) is a manifold of dimension :  $n^2 - \frac{n(n-1)}{2}  =  \frac{n(n + 1)}{2}$. 

\section{On smooth manifolds}\label{sec3}
\subsection{The $C^r$ topology on $Diff^k(M)$}\label{Cr}
Let $M$ and $N$ be two smooths manifolds and let $f\in C^r(M, N)$ with $r\leq \infty$. Let $(U,\phi)$ be a local chart on $M$ and $K$ a compact subset of $U$ such that $f(K)\subset V$ where $V$ is the domain of the chart $(V, \psi)$ on $N$. For every $l \geq 0$, we let $N^r(f, (U, \phi), (V, \psi), l, K)$ to be the set of all $g\in C^r(M, N)$ such that $g(K)\subset V$ and such that if
$ 	\bar{f}=\psi\circ f \circ \phi^{-1}$, and $\bar{g}=\psi\circ g \circ \phi^{-1} $	then  	$\parallel D^k\bar{f}(x)-D^k\bar{g}(x)\parallel \leq l,  $																										
for all $x\in \phi(K)$ and $0\leq k \leq r$.

\begin{definition}{\textbf{(Compact open $C^r-$topology)}}\cite{Hirs76}
	The sets $N^r(f, (U, \phi), (V, \psi), l, K)$ form a subbasis for a topogoly in $ C^r(M, N)$ called the \textbf{compact open $C^r$ topology}.
\end{definition}
Note that in the compact open $C^r-$topology, a neighbourhood of $f\in C^r(M, N)$ is any set which is a finite intersection of sets of the type $N^r(f, (U, \phi), (V, \psi), l, K)$.

\begin{definition}{\textbf{($C^{\infty}-$compact open topology, \cite{Hirs76})}} 
	The \textbf{$C^{\infty}$ compact open topology} is the topology induced by the inclusion  $C^{\infty}(M, N)\subset C^r(M, N)$ for $r$ finite.
\end{definition}

We defined $C^k-$compact open topology above  details on $C^0$-topology can be found in \cite{Hirs76}. When $M$ is compact, then the $C^k-$compact open topology is metrizable \cite{Hirs76} and for $k = 0$, the corresponding metric is called the \textbf{$C^0$-metric} and denoted by $d_{C^0}$.
																																																									
\subsection{Cosymplectic manifolds}																		

An almost cosymplectic structure on smooth  manifold $M$ is a pair $( \omega,\eta)$ consisting of a $2-$form  $\omega$ and a  $1-$form $\eta$  such that for each $x\in M$, the triple $ (T_xM, \omega_x, \eta_x)$ is a cosymplectic vector space.  A cosymplectic structure on $M$ is any almost cosymplectic structure $( \omega,\eta)$ on $M$ such that $d\eta = 0$ and $d\omega = 0$. We shall write  $(M, \omega, \eta)$ to mean that $M$ has a cosymplectic structure $(\omega,\eta)$. We refer to 
\cite{P-L, H-L, Tchuiaga3} and references therein. Any cosymplectic manifold $(M,\omega,\eta)$ is of odd dimension $(2n + 1)$, and is orientable with respect to the volume form 
$\eta\wedge\omega^n$. Not all odd dimensional manifolds have a cosymplectic structure \cite{Tchuiaga3}, Li \cite{H-L}. If $\Omega^1(M)$ (resp. $\chi(M)$) stands for the $C^\infty(M,\mathbb{R})-$module of all $1-$forms ( resp. all smooth vector fields) on $M$, then	the cosymplectic structure induces an isomorphism of $C^\infty(M,\mathbb{R})-$modules
\begin{equation}
	\widetilde I_{\eta, \omega}: \chi(M) \longrightarrow \Omega^1(M), X \longmapsto  \widetilde I_{\eta, \omega}(X)=\eta(X)\eta + \imath_X\omega.
\end{equation} 
The vector field $\xi:= \widetilde I_{\eta, \omega}^{-1}(\eta)$ is called the Reeb vector field of $(M, \omega,\eta)$ and is characterized by :  
$ \eta(\xi) = 1$ and $\imath_\xi\omega = 0$.

	\begin{definition}
		A map \( \tilde{J} : TM \rightarrow TM \) is called a \textbf{co-complex structure} if it satisfies the following conditions:
	\begin{enumerate}
		\item \textbf{Action on \(\ker\eta\)}:
		 Let \(X \in \ker\eta\) be a vector field such that \(\eta(X) = 0\). Then the map \(\tilde{J}\) restricted to \(\ker\eta\) should act like a complex structure, meaning:	\[
			\tilde{J}^2(X) = -X \quad \text{for all } X \in \ker\eta.
			\]
		\item \textbf{Action on \(\langle \xi \rangle\)}: Let \(\xi\) be the Reeb vector field. The map \(\tilde{J}\) should satisfy:
			\[
			\tilde{J}(\xi) = \xi.
			\]

	\end{enumerate}
	\end{definition}
	A co-complex structure \(\tilde{J}\) on \((M, \omega, \eta)\) is compatible with the cosymplectic structure if the bilinear form \(\tilde{g}\) defined by: 
	\[
	\tilde{g}(X, Y) := \omega(X, \tilde{J}(Y)) + \eta(X)\eta(Y)
	\]

	is a positive definite inner product on \(M\).

\subsection{Cosymplectomorphisms}
\begin{definition} 
	A diffeomorphism $\phi : M\longrightarrow M$ is called a  cosymplectic diffeomorphism (or cosymplectomorphism) if: $\phi^\ast(\eta)= \eta$ and $\phi^\ast(\omega)= \omega$.  
	
\end{definition}
We denote by $Cosymp_{\eta, \omega}(M)$ the group of all cosymplectomorphisms of  $(M, \omega, \eta)$. \\

Throughout the paper, we equip $ \text{Cosymp}_{\eta, \omega}(M) $ with the $C^\infty-$compact-open topology \cite{Hirs76}.
\begin{definition} 
	
	Let $(M, \omega,\eta)$ be a cosymplectic manifold. 
	A cosymplectic isotopy $\Phi =\{\phi_t\}$ is a smooth map $[0,1]\ni t\mapsto \phi_t\in Cosymp_{\eta, \omega}(M) $.
\end{definition}
We denote by $Iso_{\eta, \omega}(M)$ the group of all cosymplectic isotopies of $(M, \eta, \omega)$.

\subsection{Lagrangian-like submanifolds}	
\begin{definition}
	Let $(M,\omega,\eta)$ be a cosymplectic manifold. A submanifold $L$ of $M$ is called Lagragian-like if 
	for all $ p	\in L $ the vecror subspace $T_pL$ is a  Lagrangian-like vector subspace of  $(T_pM,\omega_p, \eta_p)$. 
\end{definition}
\begin{remark}
	In other words, a submanifold $L$ of $M$ is called Lagragian-like if  for each $p\in L$, we have  $\omega_p{_{
			\lvert T_pL}}=0$, $\eta_p{_{
			\lvert T_pL}}=0$ and $\dim T_pL=\frac{1}{2}\left( \dim T_pM - 1\right) $.
	This is equivalent to say that if  $i:L\hookrightarrow M $ is the inclusion map, then $L$ is Lagrangian-like if and only if 
	$i^*\omega=0$, 	$i^*\eta=0$ and $\dim T_pL=\frac{1}{2}\left( \dim T_pM - 1\right) $. 
\end{remark}

	\begin{example} 
	As in Example \ref{Standard-1}, let $\mathbb R^{2n + 1}$ be equipped with the global coordinate system $(x_1, \dots, x_n, y_1, \dots, y_n, z)$. The cosymplectic structure on   $\mathbb R^{2n + 1}$  is induced by the $2-$form $\omega_0 :=  \sum_{i = 1}^n  dx_i\wedge dy_i$, and $1-$form $\eta_0 : = dz$. Therefore, in the global coordinates system $((x_{i,j}),  (y_{i,j}), z_{1},\dots z_l  )$ , the 
	cosymplectic structure on  $(\mathbb R^{2n + 1})^\mathbf A$ reads: $ \eta^\mathbf A =  (\tilde\pi_{\mathbb R^{2n + 1}})^\ast(\eta_0) = \sum_{j = 1}^l d z_j,
	$ and 
	$ \omega^\mathbf A =  (\tilde\pi_{\mathbb R^{2n + 1}})^\ast(\omega_0 ) =  \sum_{j = 1}^l \left( \sum_{i = 1}^n dx_{i,j} \wedge dy_{i,j}\right).$ Assume $ l = 2c + 1$, and define the following submanifolds. The submanifold $L_1$ of $(\mathbb R^{2n + 1})^\mathbf A$ consists of elements  
	$((x_{i,j}),  (y_{i,j}), z_{1},\dots z_l  )$ such that 
	$ x_{i,j} = y_{i,j}$ for all $i,j$, and if  $ l = 2c + 1$, then $ z_{c + k} =  -z_{k}$ for $ k = 1,\dots, c$ and $z_{l} = 0$; 
	and the submanifold $L_2$ of $(\mathbb R^{2n + 1})^\mathbf A$ consists of elements 
	$((x_{i,j}),  (y_{i,j}), z_{1},\dots z_l  )$
	such that 
	$ 0= y_{i,j}$ for all $i,j$,  $ z_{c + k} =  -z_{k}$ for $ k = 1,\dots, c$ and $z_{l} = 0$. Then, $L_1$ and $L_2$ 
	are Lagrangian-like submanifolds of $(\mathbb R^{2n + 1})^\mathbf A$ of dimension 
	$nl + \frac{ l -1}{2}$.

\end{example}

\begin{example} 
		Consider a 5-dimensional cosymplectic manifold \( (M, \omega, \eta) \) with the following structure:
	 \(M = \mathbb{R}^5\),
		 \(\omega = dx_1 \wedge dx_2 + dx_3 \wedge dx_4\), and 
		 \(\eta = dx_5\). Define a 2-dimensional submanifold \( L \) of \( M \) by the following parametrization: 
	\[
	L = \{ (x_1, x_2, 0, 0, f(x_1, x_2)) \mid x_1, x_2 \in \mathbb{R} \},
	\]
	where \( f(x_1, x_2) \) is a smooth function of \( x_1 \) and \( x_2 \). At any point \( p = (x_1, x_2, 0, 0, f(x_1, x_2)) \in L \), the tangent space \( T_pL \) is spanned by the vectors: 
	\[
	\frac{\partial}{\partial x_1} + \frac{\partial f}{\partial x_1} \frac{\partial}{\partial x_5}, \quad \text{and} \quad \frac{\partial}{\partial x_2} + \frac{\partial f}{\partial x_2} \frac{\partial}{\partial x_5}.
	\]
	For vectors \( X, Y \in T_pL \), we need to check that \(\omega_p(X, Y) = 0\).
	Since \( \omega = dx_1 \wedge dx_2 + dx_3 \wedge dx_4 \), we see that for the vectors in \( T_pL \), \(\omega_p(X, Y) = 0\) because there are no components in the \(x_3\) or \(x_4\) directions. 
	The Reeb vector field \(\xi\) is defined such that \(\eta(\xi) = 1\) and \(\iota_\xi \omega = 0\). Here, \(\xi = \frac{\partial}{\partial x_5}\).
	The submanifold \( L \) does not contain the Reeb vector field, as the presence of \( x_5 \) does not violate the condition. 
	The orthogonal complement \( T_pL^{\perp} \) in \( T_pM \) must satisfy \( T_pL^{\perp} = \langle \xi \rangle \oplus T_pL \).
	For our example, \( T_pL^{\perp} \) is spanned by \(\frac{\partial}{\partial x_3}\), \(\frac{\partial}{\partial x_4}\), and any orthogonal component involving \(x_5\), and the sum \( \langle \xi \rangle \oplus T_pL \) reconstructs the tangent space \( T_pM \).
	Therefore, the submanifold \( L \) satisfies the conditions to be Lagrangian-like in the cosymplectic manifold \( (M, \omega, \eta) \). we've defined a 2-dimensional submanifold $L$ 
	in the 5-dimensional cosymplectic manifold \( (M, \omega, \eta) \)
	that is not simply a suspension of a Lagrangian submanifold from a lower dimension. For a compact Lagrangian-like, consider the 5-dimensional torus \( T^5 = S^1 \times S^1 \times S^1 \times S^1 \times S^1 \) with coordinates \( (x_1, x_2, x_3, x_4, x_5) \) where each \( x_i \) ranges over \( [0, 2\pi) \).
	Define the cosymplectic forms as follows:
	$
	\omega = dx_1 \wedge dx_2 + dx_3 \wedge dx_4,
	$ and 
	$
	\eta = dx_5.
	$
	These forms are closed and define a cosymplectic structure on \( T^5 \). 
	Define a 2-dimensional submanifold \( L \) of \( T^5 \) by: 
	\[
	L = \{ (x_1, x_2, 0, 0, k(x_1, x_2)) \mid x_1, x_2 \in [0, 2\pi) \},
	\]
	where \( k(x_1, x_2) \) is a smooth function of \( x_1 \) and \( x_2 \). 
\end{example}

The following results are motivated  by Weinstein's symplectic creed states that  " everything is a Lagrangian submanifold" (see \cite{AWein2}). 			

\begin{lemma}
	If $ \phi : (M, \omega_1, \eta_1) \rightarrow (N, \omega_2, \eta_2) $ is a  cosymplectomorphism, then the image $\phi(L)$ of a Lagrangian-like submanifold of $M$ is a Lagrangian-like submanifold of $N$.\medskip
	
\end{lemma}

\noindent \textbf{Proof.} 
The map $\phi_{\lvert L}:L\rightarrow \phi(L)$ is a diffeomorphism, $\phi(L)$ is then a submanifold of  $N$  of dimension 
$\frac{\dim M - 1}{2}$. Consider the inclusion maps $i:L\hookrightarrow M$ and $j:\phi(L)\hookrightarrow N$. We have $j=\phi \circ i\circ \phi^{-1}$, thus 
$$j^*\omega_2 	 =  (\phi \circ i \circ \phi^{-1})^*\omega_2 =
	 [(\phi^{-1})^*\circ i^*]\phi^*\omega_2 =
	  (\phi^{-1})^*(i^*\omega_1) = 0,$$ 
since $i^*\omega=0$. On the other hand, we have 
$$
	j^*\eta_2	 = (\phi \circ i \circ \phi^{-1})^*\eta_2 =
	 [(\phi^{-1})^*\circ i^*]\phi^*\eta_2  =
	  (\phi^{-1})^*(i^*\eta_1) = 0,$$ 
since $i^*\eta_1=0$. $\square$

\subsection{The graph of a closed $1-$form and Lagrangian-like submanifolds}
Let  $N$ be a smooth manifold and $T^*N$ be its cotangent bundle equipped with its standard symplectic form $\omega = d\alpha$, and let  
$\beta$ to be a closed  $1-$form on  sur $N$. 
\begin{proposition}\label{Prop0} A subspace $L \subset T^*N\times\mathbb R$  is a  Lagrangian-like submanifold   if and only if the subset $q(L)$ is countable and there exists a closed $1-$form $\beta$ on $ N$ such that $p(L) = \{(x, \beta_x): x\in N \} $  where $p :T^*N\times\mathbb R\rightarrow T^*N$ and  $ q :T^*N\times\mathbb R\rightarrow \mathbb R$ are projections.
\end{proposition}
\begin{proof}
	The set $L \subset T^*N\times\mathbb R$  is of the form $ L = (\alpha, t)$ where $\alpha$ is a section of  $ T^*N$, and then $p(L)$ is the graph of the $1-$form  $\alpha$. For $L \subset T^*N\times\mathbb R$ to be  a  Lagrangian-like submanifold  means $p^\ast(d\lambda)|_L = 0 $, and $ q^\ast(du)|_L = 0$. Let $f_0$ be a trivial linear section of the projection $p$, and derive from $p^\ast(d\lambda)|_L = 0 $ that 
	$\left( f_0|_{p(L)}\right)^\ast (p^\ast(d\lambda))|_{p(L)\times\{0\}} = 0 $, i.e., $ d\lambda|_{p(L)} = 0 $. So, $p(L)$ is a Lagrangian submanifold of $T^\ast N$. Similarly, $ q^\ast(du)|_L = 0$, implies that for each fixed $\theta_0\in T^\ast N$, with the section $s_{\theta_0}$ of $q$,  we have \\
	$\left( s_{\theta_0}|_{q(L)}\right)^\ast (q^\ast(d u))|_{\{\theta_0\}\times q(L)} = 0 $, i.e., $ d u = 0 $ for all $ u\in q(L)$ : that is,  $q(L) $ is countable. Vise-versa, assume  $L \subset T^*N\times\mathbb R$  is a submanifold  so that there exists a closed $1-$form $\beta$ on $ N$ with $p(L) = \{(x, \beta_x): x\in N \} $  and $q(L)$ is countable. For each $t\in q(L)$, let $f_t$ be the corresponding section of $p$ w.r.t to $t\in \mathbb R$. From the assumption  $d\lambda_{| p(L)}  = 0$, one derives that $p^\ast(d\lambda_{| p(L)})_{|f_t(p(L))} = 0$, for each $t$, i.e.,  $p^\ast(d\lambda)_{|p(L)\times q(L)} = 0$. Hence, $p^\ast(d\lambda)_{|L} = 0$ because $L \subseteq p(L)\times q(L)$. Since $q(L)$ is countable, we have $du_{| q(L)} = 0$, which implies $q^\ast(du)_{|p(L)\times q(L)} = 0$. Note that $  \dim (p(L))\leq \dim L \leq \dim (p(L)\times q(L))\leq   \dim (p(L)) + 0$.
\end{proof}

\begin{lemma}
	Let $\phi : (M, \omega_1, \eta_1) \rightarrow (N, \omega_2, \eta_2)$ be a diffeomorphism. Then, 
	$\phi$ is a cosymplectomorphism if and only if its graph $\Gamma_\phi=\{(p,\phi(p))\ \lvert\ p\in M\}$ is a Lagrangian-like submanifold of the cosymplectic manifold $(M\times N\times \mathbb R,
	\omega, \eta)$, with $\omega :=\pi^*_1\omega_1-\pi^*_2\omega_2, \eta :=\pi^*_1\eta_1-\pi^*_2\eta_2$, where $\pi_1$, $\pi_2$ are canonical projections from $M\times N\times \mathbb R$ onto 
	$M$, $N$ respectively.
\end{lemma}
\begin{proof}
	The map $\psi:M\rightarrow M\times N$, $p\mapsto(p,\phi(p))$ is a diffeomorphism onto
	$\psi(M)=\Gamma_\varphi$, and $\psi$ it is an immersion, hence an embedding. The graph of  $\phi$ is a submanifold of dimension  $\dim M$ of $M\times N \times \mathbb R$. Hence $\Gamma_\phi$ is Lagrangian-like, if and only if,
	$\psi^*\omega=0$ and $\psi^*\eta= 0$ ($\psi$ can be seen as an inclusion of $\Gamma_\phi$ into $M\times N\times \mathbb R$). We can compute :
	
	\begin{tabular}{rl} 
		$\psi^*\omega=$ & $\psi^*\pi^*_1\omega_1-\psi^*\pi^*_2\omega_2$ \\ $=$ & $(\pi_1\circ \psi)^*\omega_1-(\pi_2\circ \psi)^*\omega_2$ \\
		$=$ & $\omega_1-\phi^*\omega_2$ 
	\end{tabular} 
	
	since $\pi_1\circ \psi=Id_{M}$ and $\pi_2\circ \psi=\phi.$ Similarly, one shows that  $\psi^*\eta= 0.$
\end{proof}

\begin{remark}
	The graph of the identity map $id_M$ denoted $\Delta=\{(p,p)\ \lvert\ p\in M\}$ is a Lagrangian-like submanifold of  $M\times M\times \mathbb R$ called diagonal, and for all diffeomorphism  $\varphi$, the set 
	$\Gamma_\varphi \cap \Delta=\{(p,\varphi(p))\ \lvert\ \varphi(p)=p,\ p\in M\}$ consists of fixed points of $\varphi$.  
\end{remark}

\subsection{Neighborhood of a Lagrangian-like submanifold}
In this subsection, we study the local geometry of Lagrangian-like submanifolds.  

\subsection{ The normal space to a submanifold :}
Let $M$ be a smooth manifold of dimension  $n$, $X$ be a submanifold of  $M$ of dimension $k$ and  $i:X\hookrightarrow M$ 
be the inclusion of $X$ in $M$.  We know that for each $x\in M$, the tangent space $T_x X$ can be seen as a vector subspace of  $T_xM$ through the  linear map $d_xi:T_xX\hookrightarrow T_xM$. The  quotient $N_xX:=T_xM/T_xX$ is a vector space of 
dimension $(n-k)$, called the normal space to $X$ at $x$. Dual to the inclusion $di : TX \rightarrow TM$ 
there is a surjective vector bundle map,
$i^\ast = (di)^\ast : TM^\ast \rightarrow TX^\ast$. The subspace $\ker((di)^\ast)$ (i.e. the pre-image of the zero section $X\subset  T^\ast M$) consists of
co-vectors that vanish on all tangent vectors to $X$. The the co-normal bundle to $X$, denoted $NX^\ast$ is defined as $NX^\ast : =  \ker((di)^\ast)$.
\begin{definition}
	Let $M$ be a smooth manifold and  $X$ a submanifold of  $M$. The normal bundle to $X$ is the set   $NX :=\{(x,v)\ \lvert\ 
	x\in X,\ v\in N_xX\}.$
\end{definition}

The set $NX$ has the structure of a vector bundle over $X$ of rank  $n-k$ with respect to the natural projection   $\pi:NX\rightarrow X$:  a structure of smooth manifold of dimension $n$. The trivial section  $N_0=\{(x,0_x)\ \lvert\ x\in X\}$ 
of $NX$ identified to  $X$ enables us to see $X$ as a closed submanifold of  $NX$, via the embedding  $i_0:X\hookrightarrow NX$,
$x\mapsto(x,0_x)$.\medskip

\begin{definition}
	
	A  neighborhood $\mathcal{U}_0$ of the trivial section $X$ in $NX$ is said to be convex if  the intersection $\mathcal{U}_0\cap NX$ is a convex space.
\end{definition}
\begin{remark}
	
	On a
	convex neighborhood $U$ of $X$ in $NX$ we can set
	$	\psi^0(x, v) = (x, 0)$, $ \psi^1(x, v) = (x, v)$, and $\psi^t(x, v) = (x, tv)$ for $0 \leq t \leq 1$
	and get a smooth retraction $\varphi : U_0 \times [0, 1] \rightarrow U_0$ of $U_0$ onto $X$. Let $Z_t$ be the
	time-dependent vector field on $U_0$ such that $ \frac{d}{dt}\psi^t = Z_t\circ \psi^t$.
	Given a smooth form  $\alpha$ on $U_0$ 
	we have the homotopy formula : 
	$ \alpha - (\psi^0)^\ast(\alpha)  =  I d\alpha + d I\alpha,$ 
	where $I\alpha : = \int_0^1(\psi^u)^\ast\left(\iota(Z_u) \alpha\right)du$. This is the homotopy formula in a
	convex neighborhood of $X$ in $NX$. If $\alpha$  is closed and $(\psi^0)^\ast(\alpha) = 0 $, then 
	$ \alpha  =  d I\alpha$. 
	Moreover, from $\psi^t(x, 0) = (x, 0)$ for every $ t$, we derive that $Z_t|X = 0$, and hence 
	$I\alpha| X $ is trivial as well.

\end{remark}

\begin{example} {\bfseries(Construction of a convex neighborhood $U_0$ of $X$ in $NX$ and a diffeomorphism
		from $U_0$ to a neighborhood $U$ of $X$ in $M$).} \label{RMK}
	Fix a Riemann metric $g$ on $M$, the 
	vector space $N_xX$ is identified
	with the orthogonal complement $\{v \in T_xM : g(v, u) = 0, \forall u\in T_xX\}$. Consider $NX^\epsilon : = \{ (x,v) \in NX : g_x^{1/2}(v, v) < \epsilon\}$:
	a convex neighborhood of $X$ in $NX$.   The exponential map $exp: NX^\epsilon\rightarrow  M$  sends
	$(x; v)$ to $\gamma(1) $ where $\gamma  : [0; 1] \rightarrow M$  is a geodesic curve starting from $x$ whose tangent at time wero is $v$. If
	$X$ is a compact submanifold of $M$, then for $\epsilon$ small enough, the map
	$exp$ is well defined. Then exp maps $NX^\epsilon$ diffeomorphically to a neighborhood $U_0$ of $X$ in $M$, and is the identity on the zero section $X$. 
	If X is not compact, replace $\epsilon$ by a continuous map from $X$ to $\mathbb R^+$  that
	tends to zero fast enough as $x$ tends to infinity. This is the tubular
	neighborhood theorem.
\end{example}
\medskip 
We need the following well-known result from differential geometry.
\medskip 																										\begin{theorem}
	
	{\bf (Whitney's extension theorem)} ( \cite{AWein}\label{thm02}). 
	Let  $M$ be a $n-$dimensional smooth manifold, $X$ be a $k-$dimensional compact submanifold of  $M$, $NX$ the normal bundle to $X$ in $M$, $i_0:X
	\hookrightarrow NX$ the trivial section, and $i:X\hookrightarrow M$ the inclusion map. Then, there exists a convex  neighborhood  $\mathcal{U}_0$ of 
	$X$ in $NX$, a  neighborhood $\mathcal{U}$ of $X$ in $M$ and a diffeomorphism $\psi:\mathcal{U}_0\rightarrow \mathcal{U}$ such that $\psi_{\lvert X}=id_X$.
\end{theorem}

\subsection{ Normal subspace of a cosymplectic manifold.}

Assume $(M, \omega, \eta)$ is a cosymplectic manifold 
of dimension $(2n + 1)$ with Reeb vector field $\xi$, and $X$ is a Lagrangian submanifold of $M$. Then, 
$(T_xM,\omega_x, \eta_x)$ is a cosymplectic vector space  of dimension $(2n + 1)$ and  $T_xX$ a Lagrangian subspace of $ T_xM$. Define a map $ I'_x:(T_xM/T_xX)\times (T_xX)^{\omega_x, \eta_x}  \rightarrow \mathbb{R},$ as : for $[v]\in T_xM/T_xX$, and $u\in  (T_xX)^{\omega_x, \eta_x},$ we have  
$
	I'_x([v],u)= \omega(v,u) + \eta(u)\eta(v).
$
Let 
also consider the linear map $I^*_x:T_xM/T_xX \rightarrow \left(\left(  T_xX\right)^{\omega_x, \eta_x} \right) ^\ast$ defined as $ I^*_x([v])= I'_x([v],.)$. We claim that $ I^*_x$ is injective: Pick $[v]\in \ker I^*_x$, then $\omega(v,u) + \eta(u)\eta(v) =0, $ $ \forall\ u \in (T_xX)^{\omega_x, \eta_x}$, i.e. 
$v\in \left( (T_xX)^{\omega_x, \eta_x}\right)^{\omega_x, \eta_x}  = T_xX $ :  $[v]=0$. Since 
\begin{equation}
	\dim(T_xM/T_xX) = \dim T_xM-\dim T_xX = n + 1 = \dim [\left(\left(  T_xX\right)^{\omega_x, \eta_x} \right) ^\ast ],
\end{equation}
the rank theorem insure us that  $I^*_x$ is a bijection, hence an isomorphism. Thus, we have the identification 
$
	T_xM/T_xX \simeq T^*_xX\oplus \left( span \left(\xi_x \right)\right)^\ast .\square					
$

\begin{proposition} \label{thm00} 
	Let  $(M, \omega, \eta)$ be a cosymplectic manifold
	of dimension $(2n + 1)$ with Reeb vector field $\xi$, and $X$ be a Lagrangian-like submanifold of $M$. Then, the bundles  $NX$ and $T^*X \oplus span \left(\xi\right)^\ast $ 
	are canonically identified.
	
\end{proposition}
\begin{proof}
	Since $N_xX\simeq T^*_xX\oplus \left( span \left(\xi_x \right)\right)^\ast $ for all $x\in M$, then  $NX\simeq T^*X \oplus span \left(\xi\right)^\ast$.
\end{proof}

We shall need the following results in the proof of the  cosymplectic relative Moser trick. 
	\begin{lemma}\label{lem1}
	Let $M$ be a smooth manifold  of dimension $(2n +1)$. Let $Q$ be a compact sub-manifold of  $M$,  $\eta_i\in \mathcal{Z}^1(M)$, and  $\omega_i\in \mathcal{Z}^2(M)$ for $i=0,1$ such that $(\omega_0,\eta_0)$ and $(\omega_1,\eta_1)$ induce two cosymplectic structures on a neighborhood of $Q$ with $\omega_0 (q) =\omega_1(q)$ and $\eta_0 (q) =\eta_1(q)$, for all $q\in Q$. Then, on a small neighborhood of $Q$, we have that 
	\begin{itemize}
		\item $I_t(X):= \iota_X\left(\omega_0 + t(\omega_1-\omega_0)\right) + 
		\left( 	\left(\eta_0 + t(\eta_1-\eta_0)\right)(X)\right) \left(\eta_0 + t(\eta_1-\eta_0)\right) $, for all vector field $X$,	is non-degenerate,  for each  $t$.
		\item  the Lie derivative, $\mathcal{L}_{\xi_t}\omega_0$ is trivial, for each $t$. 
		\item the interior derivatives $\iota_{\xi_t}\omega_0$, and $\iota_{\xi_t}\omega_1$ are trivial, for each $t$, 
	\end{itemize}	
	where  $ \xi_t: = I_t^{-1}(\eta_t)$  with $\eta_t : = \eta_0 + t(\eta_1-\eta_0)$.
\end{lemma}
\begin{proof}
 Consider $I_t:= \iota_X\left(\omega_0 + t(\omega_1-\omega_0)\right) + 
		\left(\eta_0 + t(\eta_1-\eta_0)\right)(X)\left(\eta_0 + t(\eta_1-\eta_0)\right) $, for each  $t$  and for all vector field $X$. Since  $(\omega_0,\eta_0)$ and $(\omega_1,\eta_1)$ induce two cosymplectic structures on a neighborhood of $Q$ with $\omega_0 (q) =\omega_1(q)$ and $\eta_0 (q) =\eta_1(q)$, for all $q\in Q$, the non-degeneracy conditions of $ \widetilde  I_{\eta_0,\omega_0}$ and $ \widetilde I_{\eta_1,\omega_1}$ on the aforementioned neighborhood of $Q$ suggest  that  $I_t$
		is non-degenerate  in an open neighborhood of $Q$. Set  $ \xi_t: = I_t^{-1}(\eta_t)$, and derive  that  $ \eta_t(\xi_t)\left(1- \eta_t(\xi_t)\right) = 0$, for each $t$. Since the function $x\mapsto \eta_t(\xi_t)(x)$ is the constant function $1$ on $Q$, we can shrink the  aforementioned small neighborhood of $Q$ and assume that on this open  neighborhood the function the $x\mapsto \eta_t(\xi_t)(x)$ is positive. This implies that on such a neighborhood the only possibility for the  equation $ \eta_t(\xi_t)\left(1- \eta_t(\xi_t)\right) = 0$ to hold  is when $\eta_t(\xi_t) = 1$. Thus, the couple $(\eta_t, \omega_t)$ induces a cosymplectic structure on that neighborhood with Reeb vector field $\xi_t$.
		 To prove $\mathcal{L}_{\xi_t}\omega_0  = 0$ for each $t$, we proceed by contradiction. Assuming that there exists  $s\in ]0, 1]$ such that 
		$\mathcal{L}_{\xi_s}\omega_0  \neq 0$. In particular, since the $2-$forms $\mathcal{L}_{\xi_s}\omega_0$, and  $\mathcal{L}_{\xi_s}\omega_1$ 
		agree on $Q$,  we could have  $s\mathcal{L}_{\xi_s}\omega_1 + (1-s)\mathcal{L}_{\xi_s}\omega_0  \neq 0,$ 
		on a neighborhood  of $Q$. 
		That is, 
		\begin{equation}\label{Lie-1}
			0=  \mathcal{L}_{\xi_s}\omega_s = s\mathcal{L}_{\xi_s}\omega_1 + (1-s)\mathcal{L}_{\xi_s}\omega_0 \neq 0,
		\end{equation}
		on that neighborhood of $Q$ since $\xi_s$ is a cosymplectic vector field on a neighborhood of $Q$ \cite{Tchuiaga3}.  Formula (\ref{Lie-1})  contradicts itself.
		 For $ \omega_0(\xi_t, .) = 0 = \omega_1(\xi_t, .)$ for all $t$, we use	 similar arguments as in the proof of the second item. 

\end{proof}
\begin{theorem}[\cite{S-T-K}]\label{thm4}
	Let $M$ be a smooth compact connected manifold  of dimension $(2n +1)$, admitting two cosymplectic structures  $( \omega_0,\eta_0)$ and $(\omega_1,\eta_1)$. Assume that  $\{\eta_t\}$ (resp. $\{\omega_t\}$) is a  smooth family of closed $1-$forms (resp. closed $2-$forms) with endpoints $\eta_0$ and $\eta_1$ 
	(resp. $\omega_0$ and $\omega_1$ )  so that $(M, \omega_t,\eta_t)$ is a cosymplectic manifold for each $t$. Assume in addition that
	$\dfrac{\partial}{\partial t}[\omega_t ] =\left[\dfrac{\partial}{\partial t}\omega_t \right] =  [d\alpha_t ] = 0$, 
	 $\dfrac{\partial}{\partial t}[\eta_t ] =\left[\dfrac{\partial}{\partial t}\eta_t\right]  =[df_t] = 0$, and 
		$d(\alpha_t(\xi_t))  = 0$,  for each $t$, with $\xi_t: = \widetilde  I_{\eta_t,\omega_t}^{-1}(\eta_t)$. 
	There exists a smooth isotopy
	$\{\phi_t\}$ such that 
	  $\phi_t^\ast(\omega_t)= \omega_0$?
		 $\phi_t^\ast(\eta_t)= \eta_0 $, and 
		 $(\phi_t)_\ast(\xi_t) = \xi_0$,	for all $t\in [0,1]$.
\end{theorem}																													Theorem \ref{thm4} tells us that with respect to a certain condition on Reeb vector fields, a cosymplectic structures $(\omega,\eta)$ on a compact manifolds cannot be deformed within the cohomology classes of $\omega$, and $\eta$ to an inequivalent cosymplectic structure.

\begin{theorem}	{\bf(Relative cosymplectic Moser trick)}\label{thm01} 
	Let $M$ be a smooth connected manifold  of dimension $(2n +1)$. Let $Q$ be a compact sub-manifold of  $M$, $\eta_i\in \mathcal{Z}^1(M)$, and  $\omega_i\in \mathcal{Z}^2(M)$ for $i = 0,1$ such that $(\omega_0,\eta_0)$ and $(\omega_1,\eta_1)$ induce two cosymplectic structures on a neighborhood of $Q$ with $\omega_0 (q) =\omega_1(q)$ and $\eta_0 (q) =\eta_1(q)$, for all $q\in Q$. Then, there exist open neighborhoods $\mathcal{U}_0$ and  $\mathcal{U}_1$ of $Q$ and a local diffeomorphism $\theta:\mathcal{U}_0\longrightarrow\mathcal{U}_1$  such that $\theta^\ast(\omega_1) = \omega_0$, 
	$\theta^\ast(\eta_1) = \eta_0$ and $\theta_{|Q} = id_M$.
	
\end{theorem}
\begin{proof}
	Let $\eta_0$, $\eta_1\in \mathcal{Z}^1(M)$ and $\omega_0, \omega_1\in \mathcal{Z}^2(M)$ be as in the lemma. Firstly, we want to show that there exists a neighborhood $\mathcal{U}$ of $Q$ on which we have $ \omega_1-\omega_0 = d\alpha,$ and $ \eta_1-\eta_0 = df,$ where $\alpha$ is a $1-$form on $M$, and $f$ is a smooth function on $M$.  To do so, we shall adapt the technique used to prove a similar result in the symplectic context. 
	Fix a Riemannian metric $g$ on $M$, and identify the normal bundle $\mathbf{N}Q$ of $Q$ with the orthogonal complement $Q^{\perp}$. Then, the exponential map induces a map $exp^{\bot}: \mathbf{N}Q\longrightarrow M$ which realizes a diffeomorphism on some subset 
		$\mathbf{B}_\delta :=\big\{(q, v)\in \mathbf{N}Q: \| v\|_g<  \delta\big\},$ 
		for $\delta$ sufficiently small ($exp^{\bot}$ maps the zero section on $Q$). We set $\mathcal{U} : =  exp^{\bot}(\mathbf{B}_\delta) $ (such a $\delta > 0$ exists because $Q$ is compact), and 
		for each $0\leqslant t \leqslant 1$, we define $\phi_t(exp^{\bot}((q, v))) :=  exp^{\bot}((q, tv)).$ For $t> 0$, the map $\phi_t$ realizes a diffeomorphism 
		from $\mathcal{U}$, onto its image in  $\mathcal{U}$. Moreover, we have $ \phi_1 = id$, $\phi_0(\mathcal{U}) = Q$, and $\phi_t$ restricted to $Q$ is the 
		identity of $Q$. Put
		$\tau := \omega_1 - \omega_0,$ and 
		$\sigma := \eta_1 -\eta_0,$ and for each $0< t \leqslant 1$, define
		$Y_t(\phi_t(x)) := \dfrac{d}{d t}(\phi_t(x)),$
		for all $x\in M$. Compute
		$\phi_1^\ast(\tau) -  \phi_s^\ast(\tau) = \int_s^1 \frac{d}{d u} (\phi_u^\ast(\tau))du = d\left(\int_s^1 \phi_u^\ast (\iota(Y_u)\tau) du \right) =: d\alpha_s,$
		for any $s> 0$. 	Similarly, one obtains\\ 
		$\displaystyle \phi_1^\ast(\sigma) -  \phi_s^\ast(\sigma) = \int_s^1 \frac{d}{d u} (\phi_u^\ast(\sigma))du = d\left(\int_s^1\sigma(Y_u)\circ\phi_u du \right) =: df_s.$\\
		For vector fields $X$ and $ Z$ on $\mathcal{U}$, we have\\ 
		$ \phi_s^\ast(\tau)(X,Z) = \tau(d\phi_s(X), d\phi_s(Z))\circ \phi_s  \longrightarrow \tau(d\phi_0(X), d\phi_0(Z))\circ \phi_0 = 0,$ 
		as $s\longrightarrow 0^+$ since   $\phi_0(\mathcal{U}) = Q$, and $\tau = 0$ on each $ T_qM$,  for each $q\in Q$. This implies that 
		$\tau = d\left( (\phi^{-1}_1)^\ast(\alpha_0)\right) =: d\beta,$ and similarly, one obtains $\sigma = d \left( f_0\circ\phi^{-1}_1\right).$ 
		Recall that the restriction of $\tau$ (resp. $\sigma$) to $ T_qM$ is trivial. For instance,
		set  $\eta_t := \eta_0 + t (\eta_1 - \eta_0)$, and $\omega_t := \omega_0 + t (\omega_1 - \omega_0)$ for each $t$ and derive that,
		$\frac{\partial}{\partial t}\omega_t  = d\beta \quad \text{and}\quad  \frac{\partial}{\partial t}\eta_t = \sigma.$ The equation $ \imath_{\xi_t}\omega_t = 0$, implies that $t \imath_{\xi_t}d\beta = - \mathcal{L}_{\xi_t}\omega_0 $, for each $t$. Therefore, we apply the second item of Lemma \ref{lem1} to get  $\imath_{\xi_t}d\beta = 0 $. On the other hand, let $(\rho_t)$ be the local isotopy  generated by $\xi_t$. Using the magic Lie-Cartan formula, we compute $  \rho_t^\ast\left( \mathcal{L}_{\xi_t}\beta\right)  = - \int_0^1 \tau(Y_u, d\phi_u(\xi_t))\circ\phi_u\circ\rho_t du$, and setting,  $v_t : = \left( d_0exp^{\perp}_q\right)^{-1}(\xi_t)$, for each $t$, we differentiate the relation $\phi_t(exp^{\bot}((q, v))) =  exp^{\bot}((q, tv))$ to get $d\phi_u(\xi_t) = u\xi_t$.  Hence, $  \rho_t^\ast\left( \mathcal{L}_{\xi_t}\beta\right)  = - \int_0^1 \tau(Y_u, d\phi_u(\xi_t))\circ\phi_u\circ\rho_t du =  - \int_0^1 \tau(Y_u, u\xi_t)\circ\phi_u\circ\rho_t du = 0.$ Therefore, combining the equation  $\mathcal{L}_{\xi_t}\beta = 0$ together with $\imath_{\xi_t}d\beta = 0 $, we obtain $ d(\beta(\xi_t)) = 0$, for all $t$. Thus, we have that $\dfrac{\partial}{\partial t}\omega_t  = d\beta$, 
		$\dfrac{\partial}{\partial t}\eta_t = \sigma $,  
		and  $d(\beta(\xi_t)) = 0$, for each $t$.  Theorem \ref{thm4} together with the first item of Lemma \ref{lem1} suggests that with the latter equality constraints, we can find a smooth family of vector fields $\{X_t\}_t$ defined on a small neighborhood of $Q$ such that  
		$ I_t(X_t) + \beta + (f_0\circ\phi^{-1})\eta_t = 0,$  for each $t$, and we have $X_t = 0$ on $Q$ since the restrictions of $\sigma$ and $\tau$ to $T_qQ$ are trivial for all $q\in Q$.  
	To complete the proof,  
		it will be enough to show that there exists a small neighborhood $\mathcal{U}_0$ of $Q$ containing $\mathcal{U}$, on which the
		family of diffeomorphisms $\psi_t$ induced by the ODE:  $\dot{\psi}_t = X_t\circ\psi_t,$ and $\psi_0 = id,$ is well-defined. 
		We proceed by contradiction: Assume that no such  neighborhood exists. Then, there could exist  sequences $\{x_i\}\subset \mathcal{U}$ and $\{s_i\}\subset [0, 1)$  such that: $\psi_t(x_i)$ is defined for 
		$t\in [0, s_i]$,  $\psi_{s_i}(x_i)\in \partial \mathcal{U}$ and as $i\longrightarrow \infty$, the points $x_i$ belong to the closure of $Q$. 
		Since $Q$ is compact, a subsequence  $\{q_i\}$  of $\{x_i\}$ 
		converges to $q_0\in Q$ and  a subsequence $\{t_i\}$ of $\{s_i\}$ converges to $t_0 \in [0,1]$, whereas the orbit $[0, t_i]\ni t \mapsto\psi_t(q_i)$ converges 
		to the orbit $[0, t_0]\ni t \mapsto\psi_t(q_0)$. For each  $t\in [0, t_0]$, we have that $\psi_{t}(q_0)\in \partial \mathcal{U}$, and the compactness of  $\partial \mathcal{U}$, implies that we equally have $\psi_{t_0}(q_0)\in \partial \mathcal{U}$. This is a contradiction because $X_t$ being trivial on $Q$ could impose that $\psi_t(q_0) = q_0\in Q$, for 
		each $t\in [0, t_0]$. Therefore, we take $\mathcal{U}_1 := \psi_1( \mathcal{U}_0)\subset  \mathcal{U}$, and $\theta := \psi_1$ is the desired diffeomorphism. That is, 
		$\theta: \mathcal{U}_0\longrightarrow \mathcal{U}_1\subset  \mathcal{U}$ with $\theta^\ast(\omega_1) = \omega_0$, $\theta^\ast(\eta_1) = \eta_0$, and 
		$ \theta_{|Q} = id_Q$ because $ X_t$ is identically trivial on $Q$.
\end{proof}

The following theorem shows that near Lagrangian-like submanifolds,  all the cosymplectic
manifolds of the same dimensions are almost the same.		

\begin{theorem}	{\bf(Weinstein-like tubular neighborhood theorem)}\label{Wein-00}. 
	Let  $(M, \omega, \eta)$ be a cosymplectic manifold
	of dimension $(2n + 1)$ with Reeb vector field $\xi$, and $i : X\hookrightarrow M$ a Lagrangian-like embedding in $M$. 
	Consider $T^\ast X\times \mathbb R$ with the canonical cosymplectic structure $(\omega_0, \eta_0)$ and $X$ as  $j(X)$ 
	in $T^\ast X\times \mathbb R$ with $j : X\hookrightarrow  T^\ast X\times \mathbb R, x\mapsto(x, 0_x,0)$. Then there exist
	neigbourhoods $U_0$ of $j(X)$ in $T^\ast X\times \mathbb R$, $U_1$ of $X$ in $M$ and a diffeomorphism
	$\theta : U_0 \rightarrow U_1$ such that $ \theta\circ j = i$, $\theta^\ast \eta = \eta_0$ and  $\theta^\ast \omega = \omega_0.$
	
\end{theorem}
\begin{proof}
	Proposition \ref{thm00} implies $ NX\simeq T^*X \oplus span \left(\xi\right)^\ast\simeq  T^*X \times\mathbb R $. So,  
	we can identify the canonical cosymplectic structure $(\omega_0, \eta_0)$  as a cosymplectic structure on 
	$NX$.  By the tubular neighborhood theorem in Remark \ref{RMK}, there is a neighborhood $V$ of $X$ in $M$, a convex neighborhood $U$ of $X$ in $NX = TM/T X$, and a diffeomorphism $exp : U\rightarrow V,$ 
	determined by a Riemannian metric on $M$ such that $exp\circ j = i$. Since 
	$X$ is a  Lagrangian-like submanifold in both $TX^\ast\times \mathbb R$  and in $M$, we have 
	$exp^\ast (\omega)_{|TX} = 0 =  \omega_{0_{| TX}}  $, and $exp^\ast (\eta)_{|TX} = 0 = \eta_{0_{| TX}}$. We can use the above equalities together with Proposition \ref{thm000} to find a linear isomorphism $Lp : T_pU\rightarrow T_pU$ which is the identity map on $T_pX$, and 
	$L_p^\ast(exp^\ast (\omega_p)) = \omega_{0_{|p}}$, and $L_p^\ast(exp^\ast (\eta_p)) = \eta_{0_{|p}}$,
	that varies
	smoothly in $p\in X$. Whitney's extension theorem suggests that there
	is an embedding $ h: W \rightarrow U$, of some neighborhood $W$ of $X$ such that $h_{| X} = Id$, and $d_p h = L_p$ for every $p\in X$ : 
	$h^\ast(exp^\ast (\omega_p)) = \omega_{0_{|p}}$, and $h^\ast(exp^\ast (\eta_p)) = \eta_{0_{|p}}$. 
	At this level, we invoke  Theorem \ref{thm01}   to find a neighborhood $U_0$ of $X$ 
	in $T^\ast X\times \mathbb R$  and  a local diffeomorphism $\theta':\mathcal{U}_0\longrightarrow\mathcal{U}_1$  such that 
	$\theta'^\ast(exp^\ast (\omega) ) = \omega_0$, 
	$\theta'^\ast(exp^\ast (\eta)) = \eta_0$ and $exp\circ\theta'\circ j = i$: Take $\theta : = exp \circ \theta'$.
\end{proof}

\section{Construction of the  Weinstein-like chart}\label{sec4}

Let  $U_0$, $X$, $j$ and $U_1$ as in the above Weinstein-like tubular neighborhood theorem. Since $p :T^*X\times\mathbb R\rightarrow T^*X$ is open,  $W_0 : =  p(U_0)$ is an open subset of the zero section in $T^\ast X$, and hence 
$ W_0\times \{0\}$ is an open subset containing $j(X)$. Then, $A_0 : = \left( W_0\times \{0\} \right) \cap U_0$ is an open  subset containing $j(X)$, and $B_0 : = \theta(A_0)\subseteq U_1$  is an open  subset containing  $X$. The map $\bar \theta : = \theta_{| A_0} $ is a diffeomorphism, and  satisfies $ \bar\theta\circ j = i$, $\bar \theta^\ast \eta = \eta_0$ and  $\bar \theta^\ast \omega = \omega_0.$ In particular, since $A_0 $ is of the form $ O\times \{0\} $ where $O$ is an open neidgbourhood of the zero section in $T^\ast X$, then  the map $p\circ \bar \theta^{-1} $ induces a diffeomorphism from some  neidgbourhood $  B_0$ of $X$ in $M$ onto an open  neidgbourhood $O$ of the zero section in $T^\ast X$ with inverse map $  \bar \theta \circ f_0 $ where $f_0$ is the trivial section of the projection $p$ : compute 
\begin{equation}
	\left( \bar \theta \circ f_0\circ p\circ \bar \theta^{-1}\right)(\bar \theta(a, 0)) =  \bar \theta \circ f_0(a)  = \bar \theta (a, 0),
\end{equation}
for all $(a, 0)\in O\times \{0\},$ and  
\begin{equation}
	\left(  p\circ \bar \theta^{-1}\circ\bar \theta \circ f_0\right)  = p\circ f_0 = id_{A_0}.			
\end{equation}
If $Z(X) $ stands for the zero section in $T^\ast X$, then $ \bar \theta \circ f_{0| Z(X)} = X$ ,  $ (\bar \theta \circ f_0)^\ast (\omega) = {\omega_0}_{| O}$ and 
$ (\bar \theta \circ f_0)^\ast (\eta) = \eta_{0| \{0\}}$. $ \square$\\


\begin{remark}\label{R-1}
	Let $L_1$ and  $L_2$ be two compact  Lagrangian-like submanifolds of a cosymplectic manifold  $(M,\omega, \eta)$. Assume   $L_1$ and 
	$L_2$ are $C^1$-close, that is,  there exists a diffeomorphism $i:L_1\rightarrow L_2$ which is $C^1$-close to the inclusion map $i_0:L_1\hookrightarrow M$ of $L_1$ in $M$. The intersection $L_1 \cap L_2$ agrees with the set of 
	zero of a certain closed $1-$from on $L_1$. To see this, by 
	Theorem \ref{Wein-00}, a neighborhood  $\mathcal{U}_1$ of $L_1$ 
	in $M$ is cosymplectomorphic to a neighborhood of $j(L_1)$ in $T^*L_1\times \mathbb R$. Assume  
	$L_2$ sufficiently  $C^0$-closed to 
	$L_1$ in such a way that $L_2$ is contained in $\mathcal{U}_1$. We can suppose 
	$(M,\omega, \eta)$ is  
	$(T^*L_1\times \mathbb R,\omega_{0}, \eta_{0})$, and that $L_1$ is $j(L_1)$. If  $\theta^{-1}(L_2)$  is sufficiently $C^1$-close to $j(L_1)$ (i.e.,$\theta^{-1}(L_2)$ is a Lagrangian-like submanifold in  $T^*L_1\times \mathbb R$), then $p( \theta^{-1}(L_2))$ is the graph of a closed $1-$form  $\alpha$ on $L_1$ (by Proposition \ref{Prop0}). The zeroes of $\alpha$ are  elements of $L_1 \cap L_2$. $ \square$\\

\end{remark}

\begin{remark}\label{R-1-0}
	Let $(M, \omega, \eta)$ be a cosymplectic manifold. The set $\Delta =\{(x,x)\ \lvert \ x \in M\}$ is a  
	Lagrangian-like submanifold of  
	$(M\times M\times \mathbb R,\Omega,\eta )$, where $\Omega= \pi_1^*\omega-\pi_2^*\omega$, and $\eta : = \pi_1^*\eta-\pi_2^*\eta$.    Applying Theorem \ref{Wein-00}, we derive that there exists a neighborhood $U(\Delta)$ of 
	$\Delta$ in  $M\times M\times \mathbb R$, a neighborhood $W(j(\Delta))$ of the $j(\Delta)$ in $T^*\Delta \times \mathbb R $ and a diffeomorphism  $k:U(\Delta)\rightarrow W(j(\Delta))$
	such that $k_{\lvert \Delta}= id$, $k^*\eta_0 = \eta_{| \Delta}$ and $k^*\omega_0 = \Omega_{| \Delta}$. Consider the identifications : $\Delta \approx M$ and $T^*\Delta \approx T^*M$, and derive that 
	$k:U(\Delta)\rightarrow T^*M\times \mathbb R,\ \omega_\Delta \equiv \omega_M,\  \eta_\Delta \equiv \eta_M,\ k^*\eta_0 =\eta \ and\ k^*\omega_0=\omega.$\\ 
	In essence,  the local geometry around the diagonal $\Delta$ in $M\times M\times \mathbb R $ can be described using the familiar structure of the cotangent bundle $T^*M$ 
	augmented by $\mathbb R$. 
\end{remark}

\begin{lemma}\label{diffeo-1}
	Let $(M, \omega,\eta)$ be a closed cosymplectic manifold. Then a neighborhood $\mathcal{T}_{id_M}$ of the identity map in $\text{Cosymp}_{\eta, \omega}(M)$ can be identified with a  neighborhood  $\mathcal{S}^0_{{Z^1(M)}}$ of $0$ in the space $ Z^1(M)$ of all closed $1-$forms on $M$. 
\end{lemma}

\begin{proof}
	Consider $h$ to be a cosymplectomorphism that is sufficiently closed to the identity map in the  $C^1$-topology.  W.l.o.g we may assume $h$ is also closed enough to the identity map in the $C^0$-topology so that its graph $\Gamma_h $ is cantained in some open neighborhood  $ U(\Delta)$ of $ \Delta$  as in Theorem \ref{Wein-00}. Then,  $k(\Gamma_h)$ is a Lagrangian-like submanifold of  $T^*M\times \mathbb R$, hence $p\left( k(\Gamma_h)\right) $ is the  graph of a closed $1$-form denoted $\mathcal W_L(h)$ (by Proposition \ref{Prop0}). The correspondence $W_L:h\mapsto \mathcal W_L(h)$ is smooth from a neighborhood  $\mathcal{T}_{id_M}$ of the identity map in 
	$ \text{Cosymp}_{\eta, \omega}(M) $ into a neighborhood  $\mathcal{S}^0_{{Z^1(M)}}$ of zero  in the space of all closed $1-$forms.
	Arguing in the reverse direction, we can generate a small element in $ \text{Cosymp}_{\eta, \omega}(M) $  via any closed $1-$form sufficiently closed to zero in $ Z^1(M)$.
\end{proof}	
Lemma \ref{diffeo-1} is showcasing how small perturbations in the space of forms can induce smooth transformations in the manifold.														
\subsection{{\bf Fixed points of a cosymplectomorphism closed to the identity map}}

The fixed points of  a cosymplectomorphism  $\varphi$ are points in $\Gamma(\varphi) \cap \Delta$.

\begin{proposition}\label{fix-0}
	Let $(M,\omega, \eta)$ be a compact cosymplectic manifold with $H^1(M,\mathbb{R})=0$. Then, any cosymplectomorphism  $f:M\rightarrow M$  which is sufficiently 
	$C^1$-close to the identity map has at least two fixed points.
\end{proposition}

\noindent \textbf{Proof}.
Assume   $L_1$ and $L_2$ to be the diagonal and the  graph of  
$f$ respectively in $M\times M\times \mathbb R$. If $f$ is sufficiently $C^1$-close to the identity map, then as seen in Remark \ref{R-1}, there is a closed $1-$form  $\alpha$ on  $L_1$ whose zeroes are the elements of  $L_1 \cap L_2$; and hence the fixed points of 
$f$. Since $M$ is homeomorphic to $L_1$, we derive that $H^1(L_1,\mathbb{R})=0$, so $\alpha =dh$ for a smooth function $h$ on $L_1$. Since  $h$ attains its bounds on $L_1$,   $dh=\alpha$ has at least two critical points due to the minimum and maximum value theorem. $\square$\\

Proposition \ref{fix-0}  demonstrates that the fixed points of a map 
$f : M \longrightarrow M$   
can be found by analyzing the critical points of the smooth function 
$h$ 
on $L_1$.

\subsection{The co-flux homomorphism}
Let  $(M, \omega, \eta)$ be a compact cosymplectic manifold
of dimension $2n + 1$ with Reeb vector field $\xi$.  Following \cite{Tchuiaga3}, for each closed $1-$from $\alpha$  we don't know in general if  $X_\alpha :=  \widetilde I_{\eta, \omega}^{-1}(\alpha)$ is a cosymplectic vector field.\\ Set 
$  \mathfrak{Z}^1_\xi(M) : =\{ \alpha  \in \mathcal{Z}^1(M) : \alpha(\xi) = cte\},$
where $\xi$ is the Reeb vector field :   $\mathfrak X_{\eta, \omega}(M)$ is isomorphic to $\mathfrak{Z}^1_\xi(M) $.
Set,
\begin{equation}
	\mathbb	H^1_{Reeb}(M, \mathbb R) : = \mathfrak{Z}^1_\xi(M)/ Im(d : \mathcal{C}^0_{Reeb}(M)\longrightarrow\mathfrak{Z}^1_\xi(M) ),
\end{equation}
where $ \mathcal{C}^0_{Reeb}(M) : = \{ f\in C^{\infty}(M): \xi(f)= Cte \}.$  	 We have a well-defined surjective groups homomorphism \cite{Tchuiaga3},
\begin{equation}
	\widetilde S_{\eta,\omega}: Iso_{\eta, \omega}(M) \longrightarrow  \mathbb H_{Reeb}^{1}(M,\mathbb{R}),
	(\varphi_{t})\mapsto [\int_{0}^{1}\varphi_{t}^{*}(\widetilde{I}_{\eta, \omega}(\dot{\varphi_{t}})) dt] .
\end{equation}

\subsection*{Canonical isotopy}
Let $h$ be a cosymplectomorphism in the domain $\mathfrak D$ of the  Weinstein-like chart. 	We claim that there is a "small" isotopy contained in $\mathfrak D$ called canonical isotopy, and denoted 
$\{h_t^c\}$.  Proof of the claim: Set $ 	h_t^c : = W_L^{-1}(t\mathcal W_L(h)), $ 
$ \forall \ t\in [0,1]$.\\

We have the following fact. 
\begin{theorem}\label{Flux-W}
	Let  $(M, \omega, \eta)$ be a compact cosymplectic manifold with Reeb vector field $\xi$. 
	Let $h\in  \text{Cosymp}_{\eta, \omega}(M)$ be sufficiently  $C^1$-close to the identity map, and let $\{h_t^c\} $ be its corresponding canonical isotopy. 
	Then,  
	$[\mathcal W_L(h)] = \widetilde S_{\eta,\omega}(\{h_t^c\})$,and 	
	$\mathcal W_L(h)(\xi) = 0$.

\end{theorem}
Theorem \ref{Flux-W}  ties the geometry of the cosymplectic manifold and its diffeomorphisms to the algebraic structure of closed $1-$forms, providing a deeper understanding of their interactions. This condition ensures that the $1-$form 
	$\mathcal W_L(h)(\xi) = 0$
is compatible with the cosymplectic structure of $M$. 
The proof of Theorem \ref{Flux-W} uses the following proposition.

\begin{proposition}\label{Can-1}
	Let $(M, \omega, \eta)$ be a compact connected cosymplectic manifold. If 		$\{h_t\}$ is a cosymplectic isotopy sufficiently closed to the identity map in the $C^0-$metric, then $ \int_0^1\eta(\dot{h}_t)\circ h_tdt  = 0$.										    																											    																	\end{proposition}
\begin{proof}
	Since $\{h_t\}$ is a cosymplectic isotopy,  the function 
	$x\mapsto \left( \int_0^1\eta(\dot{h}_t)\circ h_tdt\right) (x)$ is constant \cite{Tchuiaga3}. Invoke Lemma 3.10-(2) found in \cite{S-T-3} with respect to the constant path identity and $\{h_t\}$ to derive that $| \int_0^1\eta(\dot{h}_t)\circ h_tdt|$ is uniformly controlled from above by $d_{C^0}(\{h_t\},Id) $.
	Since $d_{C^0}(\{h_t\},Id) $ can be considered as sufficiently small as we want, and the constant function $x\mapsto \left( \int_0^1\eta(\dot{h}_t)\circ h_tdt\right) (x)$  must necessarily be trivial.
\end{proof}				
\subsection*{Proof of Theorem \ref{Flux-W}}
Assume $h\in  \text{Cosymp}_{\eta, \omega}(M)$ to be sufficiently closed to the identity map in the  $C^1$-topology. 
First item. Consider the map  $\mu:M\rightarrow M\times M\times \mathbb R,\ x\mapsto (x,h(x), 0 )$, which is the embedding of the graph  $\Gamma(h)$ of $h$ in $M\times M\times \mathbb R$. By 
Theorem \ref{Wein-00}, consider a neighborhood
$U(\Delta)$ of $\Delta\times \{0\}$ in  $M\times M\times \mathbb R$, and derive the following diagram 
$M\stackrel{\mu}{\rightarrow} U(\Delta) \stackrel{k}{\rightarrow} T^*M\times \mathbb R,$ and hence  $ k\circ \mu:M \rightarrow T^*M\times \mathbb R .$ Since   $(k\circ \mu)(M)=
k(\Gamma(h)\times \{0\})$, the graph of the $1$-form $p\circ k\circ \mu$ coincides with that of the Weinstein-like $1-$form $W_L(h)$, where $p :T^*M\times\mathbb R\rightarrow T^*M$. This yields  $W_L(h)= p\circ k\circ \mu$. By the universal property of the Liouville  $1$-form  $\lambda$ on $T^*M$, we have 
\begin{equation}
	W_L(h) = \left( p\circ k\circ \mu\right)^\ast \lambda =  \mu^*\beta
\end{equation}
where  $\beta =k^*\left( p^\ast\lambda\right) $. Define a map 
$\mu_t:M\rightarrow M\times M\times\mathbb R,\ x\mapsto \mu_t(x)=(x,h_t^c(x), 0)$ : we have 
$ \mu_t\left( M\right)   \subset U(\Delta)$ et $\mu_0\left( M\right)  = \Delta\times \{0\}$ because $\mu_t(M)=\Gamma(h_t^c) \times \{0\}\subset U(\Delta)$ and  $\mu_0:x\rightarrow(x,x, 0)$. Since 
$\mu_0(M) =  \Delta\times \{0\}$ and  $\beta$ vanishes on $ \Delta\times \{0\}$, we have 
\begin{equation}
	W_L(h)=\mu^*_1 \beta - 0=\mu_1^*\beta-\mu_0^*\beta=\int_0^1\frac{d}{dt}(\mu_t^*
	\beta)dt.						
\end{equation}
This implies  
\begin{equation}\label{W-1}
	W_L(h)=\int_0^1\mu_t^*(i_{\overline{X_t}}d\beta)dt+
	d(\int_0^1\mu_t^*(i_{\overline{X_t}}\beta)dt),
\end{equation}	
where $\overline{X_t}$ is the tangent vector field along $\mu_t$.  On the other hand,  for each 
$x\in M$, the tangent space to  $M\times M\times \mathbb R$ at $(x,h_t^c(x), 0)$  decomposes as   
\begin{equation}
	T_{(x,h_t^c(x), 0)}(M\times M\times \mathbb R) = T_xM \oplus T_{h_t^c(x)}M\oplus\mathbb R ,	
\end{equation}
while the linear tangent map 
$T_x\mu_t : T_xM\rightarrow  T_xM \oplus T_{h_t^c(x)}M\oplus\mathbb R $ is defined by  
\begin{equation}
	T_x\mu_t(Y_x) =  T_xid_M\left( Y_x\right)  + T_xh_t^c\left( Y_x \right) +  0 =  Y_x  +  T_xh_t^c\left( Y_x\right) ,			
\end{equation}

$\forall Y_x \in T_xM$. If $X_t$ is the tangent vector field along $h_t^c$, then  the following decomposition holds 
$\overline{X_t}(x)=(0 + X_t(x) + 0)\in  T_xM \oplus T_{h_t^c(x)}M\oplus\mathbb R.$\\ Compute, $d\beta= k^*(p^\ast d\lambda) =\pi_2^*\omega-\pi_1^*\omega,$ on $U(\Delta)$ where $\pi_i : M_1\times M_2\times \mathbb  R\rightarrow M_i$, for $i =1,2$. Thus, for  $x\in M$ and  $Y_x \in T_xM$, we have 
\begin{eqnarray}\begin{array}{cclccccccc}
		(\mu_t^*(i_{\overline{X_t}}d\beta))(x)(Y_x)   &= &   (d\beta)_{\mu_t(x)}(\overline{X_t}(x),T_x\mu_t.Y_x), \nonumber\\
		&=&   -(\pi_1^*\omega-\pi_2^*\omega)(\mu_t(x))(\overline{X_t}(x),T_x\mu_t.Y_x), \nonumber \\
		&=& -(\pi_1^*\omega-\pi_2^*\omega)_{(x,h_t^c(x), 0)}(0 + X_t(x) + 0,Y_x +  (T_xh_t^c).Y_x),  \nonumber \\
		&=&    -\omega_x((\pi_1)_*(0 + X_t(x) + 0),(\pi_1)_*(Y_x + (T_xh_t^c).Y_x))  )\nonumber \\
		&+&\omega_{h_t^c(x)}((\pi_2)_*(0 + X_t(x) + 0),(\pi_2)_*(Y_x +  
		(T_xh_t^c).Y_x)), \nonumber \\
		& = &  -\omega_x(0,Y_x)+\omega_{h_t^c(x)}(X_t(x),(T_xh_t^c).Y_x)  , \nonumber \\
		& = &   \omega_{h_t^c(x)}(X_t(x),(T_xh_t^c).Y_x),\nonumber \\ 
		& = & (h_t^c)^*(i_{X_t}\omega)_x(Y_x).
\end{array}\end{eqnarray}
Since $(h_t^c)^*(i_{X_t}\omega)=(h_t^c)^*(i_{\stackrel{.}{h_t^c}}\omega)$, we then derive that  
\begin{equation}\label{W-2}
	(\mu_t^*(i_{\overline{X_t}}d\beta))(x)(Y_x) = (h_t^c)^*(i_{X_t}\omega)_x(Y_x) = (h_t^c)^*(i_{\stackrel{.}{h_t^c}}\omega)_x(Y_x).
\end{equation}
Thus, (\ref{W-1}) and  (\ref{W-2}) implies $
	W_L(h)=\int_0^1(h_t^c)^*(i_{\stackrel{.}{h_t^c}}\omega)dt + d(\int_0^1\mu_t^\ast(i_{\overline{X_t}}\beta)dt).
$
Hence,  $ \widetilde S_{\eta,\omega}(\{h^c_t\})= 	[ W_L(h)] + \left( \int_0^1\eta(\dot{h}_t^c)\circ h_t^c dt \right) [\eta] = 	[ W_L(h)]  ,$  by Proposition \ref{Can-1}. 
The second item follows from the formula $  W_L(h) = \left( p\circ k\circ \mu\right)^\ast \lambda$. $\square$\\


	

\end{document}